\documentclass{amsart}
\usepackage{latexsym,amscd,amssymb,amsopn,amsthm,amsfonts,amsmath}
\usepackage{color}
\usepackage[dvipsnames]{xcolor} 

\copyrightinfo{2010}{American Mathematical Society}
\newtheorem{theorem}{Theorem}[section]
\newtheorem{lemma}[theorem]{Lemma}
\newtheorem{corollary}[theorem]{Corollary}

\theoremstyle{definition}

\theoremstyle{remark}
\newtheorem{remark}[theorem]{Remark}
\numberwithin{equation}{section}

\begin{document}

\title[On the stability phenomenon of the Navier-Stokes type Equations ]
      {On the stability phenomenon  \\ of the Navier-Stokes type Equations \\ for  
			Elliptic Complexes} 
			
\author[A. Parfenov, A. Shlapunov]{Andrei Parfenov, Alexander Shlapunov}
\address{Siberian Federal University,
         Institute of Mathematics and Computer Science,
         pr. Svobodnyi 79,
         660041 Krasnoyarsk,
         Russia}

\email{ashlapunov@sfu-kras.ru, testforwisdom@mail.ru}



%\dedicatory{\it
%This paper is dedicated to Sergei Lvovich Sobolev.
%}

\begin{abstract}
Let ${\mathcal X}$ be a Riemannian $n$-dimensional smooth compact closed manifold, $n\geq 2$, 
$E^i$ be smooth vector bundles over $\mathcal X$ and  $\{A^i,E^i\}$ be an elliptic 
differential complex of linear first order operators.  We consider the operator equations, 
induced by the Navier-Stokes type equations associated with $\{A^i,E^i\}$  on the scale 
of anisotropic H\"older spaces over the layer ${\mathcal X} \times [0,T]$  with finite time 
$T > 0$. Using the properties of the differentials $A^i$ and parabolic 
operators  over this scale of spaces, we reduce the equations to a nonlinear 
Fredholm operator equation of the form 
$(I+K) u = f$, where $K$ is a compact continuous operator. It appears that the Fr\'echet 
derivative $(I+K)'$ is continuously invertible at every point of each Banach space under 
the consideration and the map $(I+K)$ is open and injective in the space. 
\end{abstract}

\maketitle

\section*{Introduction}
\label{s.0}

%Thanks in large part to S.L. Sobolev, the modern theory of partial differential equations 
%is based on the Functional Analysis and Embedding Theorems for various Banach spaces. In 
%the present paper we illustrate this thesis on the Navier-Stokes type equations for 
%elliptic complexes over H\"older spaces.

The problem of describing the dynamics of incompressible viscous fluid is of great importance 
in applications. Despite enormous efforts of many mathematicians, the existence 
theorem for classical solutions was proved for two-dimensional case only (see, for 
instance, \cite{Lera34a}, \cite{Lera34b}, \cite{Kolm42}, \cite{Hopf51}, \cite{Lady70}, 
\cite{FursVish80} among essential contributions).

In this paper we focus attention on the stability 
phenomenon for the Navier-Stokes Equations discovered by O.A. 
Ladyzhenskaya. Namely, she proved (see \cite[Theorems 10 and 11]{Lady70})  
that if for a rather regular datum there is a sufficiently regular 
unique solution to the Navier-Stokes equations then for all sufficiently 
small perturbations of the datum there are unique solutions of the same 
regularity. For infinitely smooth data and solutions this phenomenon 
was indicated in \cite{Gal13}  in the case of zero exterior forces. 
Recently, the phenomenon was verified 
for the Navier-Stokes equations on the scale of anisotropic 
H\"older spaces over the layer ${\mathbb R}^n \times [0,T]$, 
weighted at the infinity with respect to the space variables, see \cite{ShlTa18}. 

We want to investigate the stability property in the context of Navier-Stokes type 
equations associated with elliptic differential complexes, 
see \cite{MTaSh19}. Namely, let ${\mathcal X}$ be a Riemannian 
$n$-dimensional smooth compact closed manifold with metric $\mathfrak g$ and let 
$E^i$ be smooth  vector bundles over $\mathcal X$. 
Let $C^\infty_{E^i} (\mathcal X)$ denote the space of all smooth sections of the bundle 
$E^i$. Consider an elliptic complex
\begin{equation}
\label{eq.ellcomp}
   0
 \longrightarrow
   C^\infty_{E^0} ({\mathcal X})
 \stackrel{A^0}{\longrightarrow}
   C^\infty _{E^1}({\mathcal X})
 \stackrel{A^1}{\longrightarrow}
   \ldots
 \stackrel{A^{N-1}}{\longrightarrow}
   C^\infty _{E^N}({\mathcal X})
 \longrightarrow
   0
\end{equation}
of first order differential operators $A^i$ on $\mathcal X$, 
see for instance, \cite{Tark95a} or \cite[\S 10.4.3]{Nic07}. This means that 
$A^{i+1} \circ A^i \equiv 0$ and the Laplacians 
$\Delta^i = (A^{i})^* A^i +  A^{i-1} (A^{i-1})^* $ 
are the second order strongly elliptic operators at each step $i$, $0\leq i \leq N$, 
where $(A^i)^*$ is the formal adjoint differential operator for $A^i$; 
here we tacitly assume that $A_i=0$ for $i<0$ or $i>N-1$. 

If the variable $t$ enters to a section of the bundle $E^i$ as a parameter, then we 
easily may define the operator $\partial_t$ acting on sections of the induced bundle $E^i (t)$ 
over the cylinder ${\mathcal X}\times [0,+\infty)$. Then the second order operators 
$L_\mu^i = \partial _t + \mu \Delta^i$ are parabolic on ${\mathcal X}\times [0,+\infty)$ 
for each positive number $\mu$, see, for instance, \cite{Eide90}. %, \cite{Be11}. 

 Fixing two  bilinear mappings ${\mathcal M}_{i,j}$, satisfying 
\begin{equation} \label{eq.nonlinear.M}
{\mathcal M}_{i,1,x}: E^{i+1}_x \otimes E^i_x  \to E^{i}_x, \,\,
{\mathcal M}_{i,2,x}: E^i_x \otimes E^{i}_x\to E^{i-1}_x, 
\end{equation} 
at each point $x\in \mathcal X$, we set for a differentiable section $v$ of the bundle 
$E^i$:
\begin{equation} \label{eq.nonlinear}
{\mathcal N}^i (v) (x) = {\mathcal M}_{i,1,x} ((A^{i} v) (x),v(x)) + A^{i-1} 
{\mathcal M}_{i,2} (v(x),v(x))  .
\end{equation}  

In this paper we consider the following initial problem over the cylinder 
${\mathcal X}_T={\mathcal X}\times [0,T]$ with a finite time $T>0$: 
given section $f$ of the induced bundle 
$E^i (t)$ and section $v_0$ the bundle $E^i$, find a section $v$ of the induced bundle 
$E^i (t)$ and a section $p$ of the induced bundle $E^{i-1} (t)$ such that 
\begin{equation}
\label{eq.NS.complex}
\left\{
\begin{array}{lcl}
	L_\mu^i v  + {\mathcal N} ^i (v) + A^{i-1} p  =
  f & \mbox{ in } & \mathcal X\times (0,T), \\
   (A^{i-1})^*\, v  =0 , \,\,  (A^{i-2})^*\, p  =0 & \mbox{ in }  & \mathcal X\times [0,T], \\
v (x,0)= v_0 & \mbox{ in }  & \mathcal X.\\
 \end{array}
\right. 
\end{equation}
We note that the gradient operator $\nabla$, the rotation operator 
$\mbox{rot}$ and divergence operator $\mbox{div}$ represent the operators $d^i$ 
included to the de Rham complex 
and acting as the differentials between the bundles $\Lambda^i$ and $\Lambda^{i+1}$ 
of exterior differential forms over ${\mathbb R}^3$ of degrees $i$ and 
$i+1$, $i=0,1,2$, respectively. Let us express the standard non-linearity 
${\mathcal N}^1 (v) = v\cdot \nabla v$ 
in the so-called Lamb form (see \cite[\S~15]{LaLi}): 
\begin{equation} \label{eq.Lamb}
{\mathcal N}^1 (v) = v\cdot \nabla v  =(\mbox{rot} \, v)\otimes v + \nabla |v|^2/2.
\end{equation}
Taking the $3$-dimensional torus ${\mathbb T}^3$ as 
$\mathcal X$, the de Rham complex $\{d^i, \Lambda^i\}_{i=0}^2$ over it and choosing 
$i=1$ and ${\mathcal N}^1 (v) =v\cdot \nabla v $, 
 we may treat \eqref{eq.NS.complex} as the 
initial problem for the Navier-Stokes equations for incompressible fluid in the so-called 
periodic setting, with the dynamical viscosity $\mu$ of the fluid under consideration, the 
density vector of outer forces $f$, the initial velocity $v_0$ and the search-for velocity 
vector field $v$ and the 'pressure' $p$ of the flow, see, for instance, \cite{Tema95}; 
for $i=1$ we have $(d^{-1})^*\equiv 0$ and the pressure $p$ is not a subject for additional 
equation in this case. 

The first paper to consider the Navier-Stokes equations on Riemannian manifolds is the
classical paper \cite{EbinMars70} (see also %\cite{Serr59}, 
\cite{Tayl10}, \cite{ChanCzub13},  \cite{Lich16} for the development of the story, 
where the authors deal precisely with the issue of non-uniqueness for the Navier-Stokes
equations on manifolds). However, we do not discuss here relations of \eqref{eq.NS.complex} 
to the Hydrodynamics and thus,  under the imposed restrictions, 
we obtain the uniqueness of solutions to \eqref{eq.NS.complex}. Then we reduce 
\eqref{eq.NS.complex} to a nonlinear 
Fredholm operator equation  of the form $(I+K_i) u = f$ on a scale of anisotropic 
H\"older type Banach spaces over the cylinder ${\mathcal X}\times [0,T]$, where $K_i$ is a 
compact continuous operator. It appears that the Fr\'echet 
derivative $(I+K_i)'$ is continuously invertible at every point of each Banach space under 
the consideration and the map $I+K_i$ is open and injective. % in the space. 

\section{The anisotropic H\"older spaces}
\label{s.HoelderSpaces}

We begin with the definition of proper function spaces. 

Since $\mathcal X$ is a compact closed Riemannian manifold, 
choosing a volume form $dx$ on $\mathcal X$ and a Riemannian metric $(\cdot,\cdot)_{x,i}$ in 
the fibres of $E^i$, we equip each bundle $E^i$ with a smooth bundle homomorphism
   $\ast_i : E^i \to E^i{}^\ast$
defined by
   $\langle \ast_i u, v \rangle_{x,i} = (v,u)_{x,i}$ for $u, v \in E^i_x$,
and the space $C^\infty_{E^i} (\mathcal X)$  
with the unitary structure
\begin{equation*}
   (u,v)_i = \int_{\mathcal X} (u,v)_{x,i} dx
\end{equation*}
giving rise to the Hilbert space $L^2_{E^i} ({\mathcal X})$ with the norm 
$\|u\|_i = \sqrt{(u,u)_i}$. 

As usual,  we say that $(A^i)^\ast$ is the formal adjoint  differential 
operator for $A^i$ if, for all $u \in C^\infty _{E^i} ({\mathcal X}) $ and 
$v \in C^\infty _{E^{i+1}}({\mathcal X}) $, 
\begin{equation*}
  (A^i u,v)_{i+1} =(u,(A^i)^* v)_{i} .
\end{equation*}  

Besides, the Riemannian metric $\mathfrak g$  defines a natural metric structure on 
$\mathcal X$. Thus, for any smooth vector bundle $E$ over $X$ equipped with a metric 
${\mathfrak h}$ and compatible connection $\nabla$ we  may introduce the 
spaces of $s$ times continuously differentiable sections of $E$ for $s\in {\mathbb Z}_+$ 
and the H\"older spaces $C^{s,\lambda}_{E} ({\mathcal X})$ with $0<\lambda<1$, 
see, for instance, \cite[Ch. 10]{Nic07}. These are known to be Banach spaces with the 
norms:
\begin{equation*}
   \| u \|_{C^{s,\lambda}_{E} (\mathcal{X})} = \| u \|_{C^{s,0} _{E}(\mathcal{X})}
 + \mbox{sign}(\lambda) \sum_{j=0}^{s} \langle \nabla ^j u \rangle_{\lambda, \mathcal{X},E}, 
\end{equation*}
where
\begin{equation*}
   \| u \|_{C^{s,0}_{E} (\mathcal{X})}  = \sum_{j=0}^s
   \sup_{x \in \mathcal{X}}
   |\nabla ^j u (x)|, \quad
   \langle u \rangle_{\lambda,\mathcal{X},E}
 = \sup_{x,y \in \mathcal{X}, x \neq y \atop  d(x,y)\leq d_0}
   \frac{|u (x) - u (y)|}{d^\lambda (x,y)}.
\end{equation*}
Here $d(x,y)$ is the geodesic distance between points $x,y\in {\mathcal X}$, 
$d_0 = \min{(1,d_{\mathcal X})}$, $d_{\mathcal X}$ is the injectivity radius of $\mathcal X$, 
providing that the points $x,y$ can be connected by a unique minimal geodesic $\Gamma_{x,y}$, 
and, for each $\zeta\in E_x$, $\eta \in E_y$, 
\begin{equation*}
|\eta -\zeta| = |\zeta -T_{x,y}\eta |_x = |\eta -T_{y,x}\zeta |_y
\end{equation*}
with $\nabla$-parallel transport $T_{x,y}: E_y \to E_x$ along the geodesic 
$\Gamma_{x,y}$. 
In this way the space $C^{s,\lambda}_{E} ({\mathcal X})$ is independent on the 
metrics $\mathfrak g$, ${\mathfrak h}$ and the connection $\nabla$ as 
the set of sections, see  \cite[Theorem 10.2.36]{Nic07}. 
It is also known that these Banach spaces admits the standard embedding theorems.
\begin{theorem}
\label{t.emb.hoelder} 
Suppose that
   $s, s' \in \mathbb{Z}_{+}$, and
   $\lambda, \lambda' \in [0,1)$.
If $s + \lambda \geq s' + \lambda'$ then the space
   $C^{s,\lambda}_{E}({\mathcal X})$ is embedded continuously into the space
   $C^{s',\lambda'}_{E}({\mathcal X}) $.
Moreover, the embedding is compact if
   $s + \lambda > s' + \lambda'$. 
\end{theorem}
Let us introduce  the anisotropic H\"older spaces over ${\mathcal X}_T$ 
adopted to the parabolic theory  see, for instance, 
\cite{LadSoUr67}, \cite{Lady70}. Namely, for $s \in {\mathbb Z}_+$ and 
$\lambda \in [0,1)$, $\gamma \in [0,1)$ let 
$ C^{2s,\lambda,s,\gamma}_{E} (\mathcal{X}_T)$
be the space of sections of the induced bundle $E (t)$ over $\mathcal{X}_T$ with continuous 
partial derivatives $\nabla^m_x \partial^j_t u$, for $m + 2j \leq 2s$ and 
the finite norm   
\begin{equation*}
   \| u \|_{C^{2s,\lambda,s,\gamma} _{E}(\mathcal{X}_T)}
 =\sum_{m + 2j \leq 2s}
 \sup_{t \in [0,T]}  \|   \partial^j_t \nabla^m_x u (\cdot,t)
   \|_{C^{0,\lambda}_{E} (\mathcal{X})} + \gamma 
\sum_{m + 2j \leq 2s} \langle \nabla^m_x \partial^j_t u\rangle
_{\gamma,{\mathcal X}_T,E}  .
\end{equation*}
where, for $\gamma>0$, 
\begin{equation*}
\langle u\rangle_{\gamma,{\mathcal X}_T,E} 
=\sup_{t', t'' \in [0,T] \atop t' \neq t''}
   \frac{\|u (\cdot,t') - u (\cdot,t'') \|_{C^{0,0}_{E} (\mathcal{X})}}{|t'-t''|^{\gamma}}.
\end{equation*}
We also need a function space whose structure goes slightly beyond the scale of 
function spaces
$C^{2s,\lambda,s,\gamma}_{E} (\mathcal{X}_T)$.
Namely, given any integral $k \geq 0$, we denote by
 $C^{2s+k,\lambda,s,\gamma}_{E} (\mathcal{X}_T)$
Namely, given any $k\in {\mathbb Z}_+$, we denote by
 $C^{2s+k,\lambda,s,\gamma}_{E} (\mathcal{X}_T)$
the space of all continuous functions $u$ on $\mathcal{X}_T$ 
with  $\nabla^l_x u$ belonging to $C^{2s,\lambda,s,\gamma}_{E} (\mathcal{X}_T)$ for all 
$l\in {\mathbb Z}_+$ satisfying $0\leq l \leq k$. This is a Banach  
space  with the norm
\begin{equation*}
   \| u \|_{C^{2s+k,\lambda,s,\gamma}_{E} (\mathcal{X}_T)}
 = \sum_{l=0}^k
   \| \nabla^l_x u \|_{C^{2s,\lambda,s,\gamma }_{E} (\mathcal{X}_T)}.
\end{equation*}
As it is customary in the parabolic theory, 
we use these spaces for $\gamma=0$ and $\gamma=\frac{\lambda}{2}$, only. The following embedding theorem is rather expectable.
\begin{theorem}
\label{t.emb.hoelder.t}
Let  
   $k,s, s' \in \mathbb{Z}_{+}$,
      $\lambda, \lambda' \in [0,1)$.
If    $s+\lambda \geq s'+\lambda'$ 
then the space
   $C^{2s+k,\lambda,s,\frac{\lambda}{2}}_{E} (\mathcal{X}_T)$ is embedded continuously into
   $C^{2s'+k,\lambda',s,\frac{\lambda'}{2}}_{E} (\mathcal{X}_T)$.
The embedding is compact if $s + \lambda > s' + \lambda'$.
\end{theorem}

We also need the following expectable lemmata.

\begin{lemma}
\label{l.diff.oper}
Suppose that $s, k \in {\mathbb Z}_+$  and $\lambda \in [0,1)$.
Then it follows that
\begin{enumerate}
\item
any differential operator of order $k'\leq k$ on $\mathcal{X}$  acting 
between vector bundles $E$ and $F$ maps
    $C^{2s+k,\lambda,s,\frac{\lambda}{2}}_{E} (\mathcal{X}_T)$
continuously into
    $C^{2s+k-k',\lambda,s,\frac{\lambda}{2}}_{F} (\mathcal{X}_T)$
\item
if $0 \leq j \leq s$ then the operator $\partial_t^j$ maps
   $C^{2s+k,\lambda,s,\frac{\lambda}{2}}_{E} (\mathcal{X}_T)$ 
continuously into 
   $C^{2(s-j)+k,\lambda,s-j,\frac{\lambda}{2}}_{E} (\mathcal{X}_T)$;
\item
operator $ \Delta^i$ maps $C^{2(s+1)+k,\lambda,s+1,\frac{\lambda}{2}}_{E^i} (\mathcal{X}_T)$
continuously into $C^{2s+k,\lambda,s,\frac{\lambda}{2}}_{E^i} (\mathcal{X}_T)$. 
\end{enumerate}
\end{lemma}

In the sequel we will always assume that there are constants $c_{i,j} ({\mathcal M})$
such that 
\begin{equation} \label{eq.property.M1}
|{\mathcal M}_{i,1,x}  (v, u)|\leq c_{i,1} (\mathcal M) |u|\, |v|, \,\, 
|{\mathcal M}_{i,2,x}  (w, u)|\leq 
c_{i,2} (\mathcal M) |u|\, |w|
\end{equation}
for all $x\in \mathcal X$ and all $v\in E^{i+1}_x$ and $ u,w \in E^{i}_x$. 
 
\begin{lemma}
\label{l.product}
Let $s,k \in {\mathbb Z}_+$ and $\lambda \in [0,1)$. If \eqref{eq.property.M1} 
holds then forms  \eqref{eq.nonlinear.M} induce continuous bilinear operators
\begin{equation} \label{eq.M1}
 {\mathcal M}_{i,j} : C^{2s+k,\lambda,s,\frac{\lambda}{2}} _{E^\cdot} ({\mathcal X}_T)  
\times C^{2s+k,\lambda,s,\frac{\lambda}{2}} _{E^i} ({\mathcal X}_T) \to 
C^{2s+k,\lambda,s,\frac{\lambda}{2}} _{E^\cdot} ({\mathcal X}_T) ,
\end{equation}
satisfying 
  $ \| {\mathcal M}_{i,j}(v, u) \|_{C^{2s+k,\lambda,s,\frac{\lambda}{2}} 
_{E^\cdot} (\mathcal{X}_T)} \leq
  c\,
   \| u \|_{C^{2s+k,\lambda,s,\frac{\lambda}{2}} _{E^i}(\mathcal{X}_T)}
   \| v \|_{C^{2s+k,\lambda,s,\frac{\lambda}{2}} _{E^\cdot} (\mathcal{X}_T)}
$
with a constant  $c > 0$ independent of $u$ and $v$.
\end{lemma}

\begin{proof} Indeed, since ${\mathcal M}_{i,j,x}$ are bilinear forms, 
\begin{equation} \label{eq.M.diff}
{\mathcal M}_{i,j}  (v,u ) - {\mathcal M}_{i,j} (v^{(0)}, u^{(0)}) = 
\end{equation} 
\begin{equation*}
{\mathcal M}_{i,j} (v-v^{(0)},u^{(0)}) + 
{\mathcal M}_{i,j} (v^{(0)},u-u^{(0)}) + {\mathcal M}_{i,j} (v-v^{(0)},u-u^{(0)}),
\end{equation*} 
at each points $v,v^{(0)},u,u^{(0)}$ of the Banach spaces under the consideration. Then 
Lemma \ref{l.diff.oper} implies  that \eqref{eq.M1}  
are continuous operators because of \eqref{eq.property.M1}, \eqref{eq.M.diff}.
\end{proof}

\section{Elliptic complexes over the H\"older spaces}
\label{s.NS.deRham}

The behaviour of the elliptic complexes on the H\"older scale
 is well known (see \cite[Ch. 2, \S 2.2]{Tark95a}, \cite[\S 10.4.3]{Nic07}. 
Namely, consider the bounded linear operator
\begin{equation}
\label{eq.Hoelder.Laplace}
   \Delta^i :\, C^{s+2,\lambda}_{E^i} (\mathcal{X}) \to
              C^{s,\lambda}_{E^i} (\mathcal{X})
\end{equation}
induced by the Laplacian $\Delta^i$. 
Let $\mathcal{H}^i$ stand for the so-called `harmonic space' of complex 
\eqref{eq.ellcomp}, i.e. 
\begin{equation*}
\mathcal{H}^i = \{ u \in C^\infty _{E^i}(\mathcal{X}): A^i u =0 \mbox{ and }
(A^{i-1})^* u =0 \mbox{ in } {\mathcal X}\}.
\end{equation*}
Denote by $\Pi^i $ the orthogonal projection from $L^2 _{E^i}(\mathcal{X})$ 
onto $\mathcal{H}^i$. 
 
\begin{theorem}
\label{t.Hoelder.Laplace}
Let $0\leq i \leq N$, $s\in {\mathbb Z}_+$, $0<\lambda<1$. 
Then operator \eqref{eq.Hoelder.Laplace} is Fredholm: 
\begin{enumerate}
\item
the kernel of operator \eqref{eq.Hoelder.Laplace} equals to 
the finite-dimensional space $\mathcal{H}^i$;
\item
given $v\in C^{s,\lambda}_{E^i} (\mathcal{X})$ there is a form 
$u\in C^{s+2,\lambda}_{E^i} (\mathcal{X})$ such that $\Delta^i u =v$ 
if and only if $(v,h)_{i} =0$
for all $h \in \mathcal{H}^i$; 
\item
there exists a pseudo-differential operator $\varphi^i$ on $\mathcal{X}$ 
such that the operator 
\begin{equation*}
   \varphi^i :\, C^{s,\lambda}_{E^i} (\mathcal{X}) \to
              C^{s+2,\lambda}_{E^i} (\mathcal{X}),
\end{equation*}
induced by $\varphi^i$, is linear bounded and with the identity $I$ we have 
\begin{equation}
\label{eq.Hodge.0}
  A^i \Pi^i =0, \, (A^{i-1})^*\Pi^i  =0 ,   
\, \Pi^{i+1} A^i =0, \, \Pi^{i-1} (A^{i-1})^* 
\end{equation}
\begin{equation*}
\Pi^i \circ \Pi^i = \Pi^i, \,
\varphi^i \Pi^i =0, \,  \Pi^i \varphi^i v =0,    
\end{equation*}
\begin{equation*}
\varphi^i \Delta ^i   =  I- \Pi^i  \mbox{ on } C^{s+2,\lambda}_{E^i} (\mathcal{X}),  
\,\, 
 \Delta ^i \varphi^i   = I - \Pi^i \mbox{ on } C^{s,\lambda}_{E^i} (\mathcal{X}). 
\end{equation*}
\end{enumerate}
\end{theorem}
\begin{proof} For $C^\infty$-smooth sections see, for instance, 
\cite[Theorem 2.2.2]{Tark95a} or \cite[Theorem 10.4.29]{Nic07}. 
For the extension to the H\"older spaces we refer to the standard procedure 
using  apriori estimates for elliptic operators, see, for instance, 
\cite[Ch. 4--6]{GiTru83} or \cite[Theorem \S 10.3, 10.4]{Nic07}. 
\end{proof}
Next, for a differential operator $A$ acting on sections of the vector bundle
$E^i$ over $\mathcal{X}$, we denote by
   $C^{s,\lambda}_{E^i} (\mathcal{X}) \cap \mathcal{S}_A$
the space of all the sections 
   $u \in C^{s,\lambda}_{E^i}  (\mathcal{X})$
satisfying $Au = 0$ in the sense of the distributions in $\mathcal{X}$.
This space is obviously a closed subspace of
   $C^{s,\lambda}_{E^i}  (\mathcal{X})$
    and so this is a Banach space under the induced norm.

\begin{corollary} \label{c.Hoelder.d}
Let $s \in {\mathbb Z}_+$, $0<\lambda < 1$. Differential 
complex \eqref{eq.ellcomp} induces continuous linear operators
\begin{equation} 
 \label{eq.Hoelder.d}
A^i \oplus (A^{i-1})^* : C^{s+1,\lambda}_{E^i} (\mathcal{X}) 
  \to C^{s,\lambda}_{E^{i+1}} (\mathcal{X}) \cap \mathcal{S}_{A^{i+1}} 
	\times C^{s,\lambda}_{E^{i-1}} (\mathcal{X}) \cap \mathcal{S}_{(A^{i-2})^*} . 
\end{equation}
These operators are Fredholm. More precisely,  
\begin{enumerate}
\item 
the kernel of  \eqref{eq.Hoelder.d} coincide with the finite-dimensional space 
$\mathcal{H}^i$;
\item 
the (closed) range of operator \eqref{eq.Hoelder.d} consists of all pairs
$(f,g) \in  C^{s,\lambda}_{E^{i+1}} (\mathcal{X}) \cap \mathcal{S}_{A^{i+1}}
\times C^{s,\lambda}_{E^{i-1}} (\mathcal{X}) \cap \mathcal{S}_{(A^{i-2})^*} $, 
satisfying for all  $\hat h \in \mathcal{H}^{i-1} $ and 
 all  $h \in \mathcal{H}^{i+1}$
\begin{equation*}
 (f ,h)_{i+1}  + (g ,\hat h)_{i-1} = 0 .
\end{equation*}
\end{enumerate}
\end{corollary}

\begin{proof} It follows from Theorem \ref{t.Hoelder.Laplace} 
that the kernel of operator \eqref{eq.Hoelder.d}, 
coincides  with the  space $\mathcal{H}^i$. Moreover  
Theorem \ref{t.Hoelder.Laplace} implies the following simple lemma.

\begin{lemma} 
\label{l.Phi}
Let $s \in {\mathbb Z}_+$, $\lambda\in (0,1)$.  
The pseudo-differential operators $\varPhi_i =(A^i)^* \varphi^{i+1} $, 
$\hat \varPhi^i =A^{i} \varphi^{i} $ 
on $\mathcal{X}$ induce continuous maps
\begin{equation*}
   \varPhi_i :\, C^{s,\lambda} _{E^{i+1}}(\mathcal{X}) \to
 C^{s+1,\lambda}_{E^i} (\mathcal{X}) \cap \mathcal{S}_{(A^{i-1})^*} ,\,\,
   \hat \varPhi^i :\, C^{s,\lambda} _{E^{i}}(\mathcal{X}) \to
 C^{s+1,\lambda}_{E^{i+1}} (\mathcal{X}) \cap \mathcal{S}_{A^{i+1}} ,
\end{equation*}
satisfying
\begin{equation}\label{eq.Hodge.1}
\varPhi_i \Pi^{i+1} =0, \,  \Pi^{i} \varPhi_i  =0, \,\, 
\hat \varPhi^i \Pi^{i} =0, \,  \Pi^{i+1} \hat \varPhi^i  =0,  
\end{equation}
\begin{equation}\label{eq.Hodge.2}
\varPhi_i A ^i  u  + 
A^{i-1} \varPhi_{i-1} u   =u- \Pi^i u
\mbox{ if } u \in C^{s,\lambda}_{E^{i}} (\mathcal{X}), \, A^i u \in 
C^{s,\lambda}_{E^{i+1}} (\mathcal{X}),
\end{equation}
\begin{equation}\label{eq.Hodge.3}
\hat \varPhi_{i-1} (A ^{i-1})^*  v  + 
(A^{i})^*\hat \varPhi_{i} v   =v- \Pi^i v, \mbox{ 
if } v \in C^{s,\lambda}_{E^{i}} (\mathcal{X}), \,\, (A^{i-1})^* v \in 
C^{s,\lambda}_{E^{i-1}} (\mathcal{X}).
\end{equation}
\end{lemma}

\begin{proof} Indeed, by the definition of the complex \eqref{eq.ellcomp},
\begin{equation} \label{eq.Hodge.4}
A^i \Delta^i = A^i (A^i)^* A^i = \Delta^{i+1} A^i, \, 
(A^{i-1})^* \Delta^i = (A^{i-1})^* A^{i-1} (A^{i-1})^*  = \Delta^{i-1} (A^{i-1})^*.
\end{equation}
Thus, we conclude that, on the sections with sufficient differentiability, 
\begin{equation} \label{eq.Hodge.5}
A^i \varphi^i = \varphi^{i+1} A^i, \, 
(A^{i-1})^* \varphi^i  = \varphi^{i-1} (A^{i-1})^*,
\end{equation}
and the statement follows from 
Theorem \ref{t.Hoelder.Laplace}. 
\end{proof}
This lemma proves the statement on the range of  operator \eqref{eq.Hoelder.d}. 
\end{proof}

This corollary just reflects the well-known fact that the space $\mathcal{H}^i$ represents 
$i$-th cohomologies of complex \eqref{eq.ellcomp} 
over the scale $C^{s,\lambda}_{E^i} (\mathcal{X})$, see \cite[Ch. 2]{Tark95a}. 

Now we start to study complex \eqref{eq.ellcomp} over the scale
   $C^{2s+k, \lambda,s,\frac{\lambda}{2}}_{E^i} (\mathcal{X}_T)$.
As elliptic operators are not fully consistent with the parabolic dilation principle on
   $\mathcal{X}_T $, 
we should expect some loss of regularity of solutions to the elliptic system 
\begin{equation*} 
   A^i u = f, \,\, (A^{i-1})^* u = g, 
\end{equation*}
 on this scale of function spaces.

Again, for a differential operator $A$ acting on sections of 
 the induced vector bundle $E(t)$ over $\mathcal{X}$,
we write $C^{2s+k, \lambda,s,\gamma}_{E} (\mathcal{X}_T) \cap \mathcal{S}_A$
for the space of all sections of $E$ over $\mathcal{X} $ of the class 
$C^{2s+k, \lambda,s,\gamma}_{E} (\mathcal{X}_T)$ satisfying  in the sense of distributions 
\begin{equation*}
Au\, (\cdot,t) = 0 \mbox{ in } {\mathcal X}
\mbox{ for all fixed } t \in  [0,T].
\end{equation*} 
This space is obviously  a closed subspace of
      $C^{2s+k, \lambda,s,\gamma}_{E} (\mathcal{X}_T)$, 
 and so it is a Banach space under the induced norm.

Actually, we want to extend Corollary \ref{c.Hoelder.d} to 
operator $   A^i \oplus (A^{i-1})^*$ on the anisotropic scale
$C^{2s+k, \lambda,s,\frac{\lambda}{2}}_{E^i} (\mathcal{X}_T)$. 
Similarly to the scale $C^{s,\lambda}_{E^i} (\mathcal{X})$, we use the potentials  
$   \varPhi_i ,\hat \varPhi_i  $ on sections of the induced bundles over $\mathcal{X}_T$.
The variable $t$ enters into the potentials $ (\varPhi_i  f) (x,t)$, 
$ (\hat \varPhi_i  g) (x,t)$ as a parameter and 
the pair $(x,t)$ is assumed to be in the  layer $\mathcal{X}_T $. 
However, elements of the space $C^{2s+k, \lambda,s,\frac{\lambda}{2}}_{E^i} 
(\mathcal{X}_T)$ have additional smoothness with respect to $t$ that can not 
be improved by the potentials $   \varPhi_i ,\hat \varPhi_i  $. 
To avoid this difficulty,  we introduce
$C^{2s+k, \lambda,s,\frac{\lambda}{2}}_{E^{i}} (\mathcal{X}_T)\cap \mathcal{D}_{A^i}$ 
to be the space of all sections $u$ from
$C^{2s+k, \lambda,s,\frac{\lambda}{2}}_{E^{i}} (\mathcal{X}_T) $
with the property that
   $A^i u \in C^{2s+k, \lambda,s,\frac{\lambda}{2}}_{E^{i+1}} (\mathcal{X}_T)$.
We endow this space with the so-called graph norm
\begin{equation*}
 \| u \|_{C^{2s+k, \lambda,s,\frac{\lambda}{2}}_{E^{i}} (\mathcal{X}_T)}
 + \| A^i u\|_{C^{2s+k, \lambda,s,\frac{\lambda}{2}}_{E^{i+1}} (\mathcal{X}_T)}.
\end{equation*}

\begin{lemma} \label{eq.Ai.Dom}
Suppose that $k,s \in {\mathbb Z}_+$, $0 < \lambda < 1$. 
Then the spaces 
\begin{equation*}
C^{2s+k, \lambda,s,\frac{\lambda}{2}}_{E^{i}} (\mathcal{X}_T)\cap \mathcal{D}_{A^i},  
C^{2s+k, \lambda,s,\frac{\lambda}{2}}_{E^{i+1}} (\mathcal{X}_T)\cap \mathcal{D}_{(A^{i})^*},
 C^{2s+k, \lambda,s,\frac{\lambda}{2}}_{E^{i}} (\mathcal{X}_T)\cap \mathcal{D}_{A^i \oplus 
(A^{i-1})^*}
\end{equation*}
are Banach spaces and the operator    $A^i \oplus (A^{i-1})^*$ maps boundedly as 
\begin{equation}
\label{eq.Hoelder.d.t}
C^{2s+k, \lambda,s,\frac{\lambda}{2}}_{E^{i}} (\mathcal{X}_T)\cap \mathcal{D}_{A^i 
\oplus (A^{i-1})^*}
  \to C^{2s+k, \lambda,s,\frac{\lambda}{2}}_{E^{i+1}} 
(\mathcal{X}_T) \cap \mathcal{S}_{A^{i+1}}  \times 
C^{2s+k, \lambda,s,\frac{\lambda}{2}}_{E^{i-1}} 
(\mathcal{X}_T) \cap \mathcal{S}_{(A^{i-2})^*}.
\end{equation}
\end{lemma}

\begin{proof} If $\{ u_\nu \}$ is a Cauchy sequence in
  $ C^{2s+k, \lambda,s,\frac{\lambda}{2}}_{E^{i}} (\mathcal{X}_T)\cap \mathcal{D}_{A^i}$,
then it is a Cauchy sequence in the space 
   $C^{2s+k, \lambda,s,\frac{\lambda}{2}}_{E^{i}} (\mathcal{X}_T)$
and
   $\{ A^i u_\nu \}$
is a Cauchy sequence in the space
   $C^{2s+k, \lambda,s,\frac{\lambda}{2}}_{E^{i+1}} (\mathcal{X}_T)$. 
As the spaces are complete we conclude that the sequence $\{ u_\nu \}$ converges in
   $C^{2s+k, \lambda,s,\frac{\lambda}{2}}_{E^{i}} (\mathcal{X}_T)$
to an element $u$ and the sequence $\{ A^i u\} $ converges in
   $C^{2s+k, \lambda,s,\frac{\lambda}{2}}_{E^{i+1}} (\mathcal{X}_T)$
to an element $f$. Obviously, $A^i u = f$ is fulfilled in the sense of distributions and 
thus $A^{i+1} f=0$ in the sense of distributions because of $A^{i+1}\circ A^i\equiv 0$. 
Hence, $u$ belongs to
   $C^{2s+k, \lambda,s,\frac{\lambda}{2}}_{E^{i}} (\mathcal{X}_T)
\cap \mathcal{D}_{A^i}$ 
and it is the limit of the sequence $\{ u_\nu \} $ in this space. 
Thus, we have proved that the space 
   $C^{2s+k, \lambda,s,\frac{\lambda}{2}}_{E^{i}} (\mathcal{X}_T)
\cap \mathcal{D}_{A^i}$ is  a Banach space. 
The proof for other two spaces is similar. 
Moreover, by the very definition of the space, the operator $A^i\oplus (A^{i-1})^*$ induces a
continuous linear operator as is shown in \eqref{eq.Hoelder.d.t}. 
\end{proof}

Now, let  $C^{s,\gamma} ([0,T], \mathcal{H}^i)$ stands 
for the set of sections of the induced bundle $E^i (t)$ over $\mathcal{X}_T$ of the form 
\begin{equation*} 
u (x,t) = \sum_{q=1}^{\dim(\mathcal{H}^i)} c_q (t) b_q (x)
\end{equation*}
where $c_q \in  C^{s,\gamma} [0,T]$ and $\{ b_q  \}_{q=1}^{\dim(\mathcal{H}^i)}$ 
is an $L^2 _{E^i}(\mathcal{X})$-orthonormal basis in $\mathcal{H}^i$.

\begin{corollary}
\label{c.Hoelder.d.t}
Suppose that $k,s \in {\mathbb Z}_+$, $0 < \lambda < 1$. Operator \eqref{eq.Hoelder.d.t}
has closed range consisting of all pairs
$(f,g) \in  C^{2s+k, \lambda,s,\frac{\lambda}{2}}_{E^{i+1}} 
(\mathcal{X}_T) 	\cap \mathcal{S}_{A^{i+1}} 	
\times C^{2s+k, \lambda,s,\frac{\lambda}{2}}_{E^{i-1}} 
(\mathcal{X}_T) \cap \mathcal{S}_{(A^{i-2})^*} 	$,  satisfying 
for all  $h \in \mathcal{H}^{i+1}$, all $\hat h \in \mathcal{H}^{i-1}$ 
 and all $t\in [0,T]$
\begin{equation}\label{eq.Hodge.t}
 (f (\cdot,t) ,h)_{i+1} +  (g (\cdot,t) ,\hat h)_{i-1} = 0.
\end{equation}
The kernel of operator \eqref{eq.Hoelder.d.t} equals to 
$C^{s,\frac{\lambda}{2}} ([0,T], \mathcal{H}^i)$.
\end{corollary}

\begin{proof} We begin with the following lemma.

\begin{lemma} 
\label{l.range.Pi}
The pseudo-differential operator $\Pi^i $ induces continuous maps 
\begin{equation} \label{eq.Pi}
\Pi^i :\, C^{2s+k, \lambda,s,\frac{\lambda}{2}}_{E^{i}} (\mathcal{X}_T) \cap
\mathcal{D}_{A^i} \to C^{2s+k', \lambda,s,\frac{\lambda}{2}}_{E^{i}} (\mathcal{X}_T)\cap 
\mathcal{D}_{A^i} 
\end{equation}
for any $k'\in \mathbb N$. The range of operator \eqref{eq.Pi} %, \eqref{eq.Pi.*} 
coincides with $C^{s,\frac{\lambda}{2}} ([0,T], \mathcal{H}^i)$.  
\end{lemma}

\begin{proof} Indeed, as the space 
$C^{2s+k, \lambda,s,\frac{\lambda}{2}}_{E^i} (\mathcal{X}_T) $ 
is continuously embedded to $L^2 _{E^i} (\mathcal{X})$ we see that 
\begin{equation*}
(\Pi^i u ) (x,t) = \sum_{q=1}^{\dim(\mathcal{H}^i)} (u (\cdot,t), b_q)_i \, b_q (x)
\end{equation*}
for each $u\in C^{2s+k, \lambda,s,\frac{\lambda}{2}}_{E^i} (\mathcal{X}_T) $. Set $c_q (t) =
(u (\cdot,t), b_q)_i $.  Then 
\begin{equation*}
\| (d^j c_q (t))/dt^j\|_{C^{0,0} [0,T]} = \sup_{t\in [0,T]}
|(\partial_t ^j u (\cdot,t), b_q)
_{i}|\leq C_\mathcal{X} 
\| \partial_t^j u \|_{C^{0,0,0,0}_{E^i}(\mathcal{X}_T)}
\end{equation*}
for each $0\leq j \leq s$, i.e. $\Pi^i u\in  C^{s,0} ([0,T], \mathcal{H}^i)$. 
Moreover, 
\begin{equation*}
\langle (d^jc_q)/dt^j\rangle_{\frac{\lambda}{2},[0,T]} = 
\langle (\partial_t ^j u (\cdot,t), b_q)
_{i} \rangle \leq C_\mathcal{X} 
\|  u \|_{C^{2s+k, \lambda,s,\frac{\lambda}{2}}_{E^{i}} (\mathcal{X}_T)}
\end{equation*}
for each $0\leq j \leq s$, i.e. $\Pi ^i 
u\in  C^{s,\frac{\lambda}{2}} ([0,T], \mathcal{H}^i)$.

If the section $v$ belongs to $ C^{s,\frac{\lambda}{2}} ([0,T], \mathcal{H}^i)$ 
then, obviously, $\Pi ^i v =v$. Moreover, 
\begin{equation}
\label{eq.Pi.v}
\|v \|_{C^{2s+k', \lambda,s,0}_{E^i} (\mathcal{X}_T)} \leq c\,
\sum_{q=1}^{\dim(\mathcal{H}^i)} \| c_q (t)\|_{C^{s,0} ([0,T])}
\| b_q \|_{C^{2s+k',\lambda}_{E^i} (\mathcal{X})},
\end{equation}
with a constant $c>0$ independent on $v$, 
i.e. $v \in C^{2s+k,\lambda,s,0}_{E^i} (\mathcal{X}_T)$. 
Moreover, as $c_q \in  C^{s,\frac{\lambda}{2}}[0,T]$ then, for all 
$j\leq s$
\begin{equation*}
\sup_{t', t'' \in [0,T] \atop t' \neq t''}
   \frac{\|\partial_{t'}^j v (\cdot,t') - 
	\partial_{t''}^j  v (\cdot,t'') \|_{C^{2(s-j)+k',0}_{E^i} (\mathcal{X})}}{|t'-t''|
	^{\frac{\lambda}{2}}}\leq
\end{equation*}
\begin{equation*}
\sum_{q=1}^{\dim(\mathcal{H}^i)}\sup_{t\in [0,T]} 
\langle (d^j c_q (t))/dt^j\rangle_{\frac{\lambda}{2}, [0,T]} 
\| b_q \|_{C^{2(s-j)+k',0} _{E^i}(\mathcal{X})},
\end{equation*}
i.e. $v \in C^{2s+k',\lambda, s,\frac{\lambda}{2}}_{E^i} (\mathcal{X}_T)$. 
This proves that the range of  operator  \eqref{eq.Pi} coincides 
with the space $C^{s,\frac{\lambda}{2}} ([0,T], \mathcal{H}^i)$.

As $\Pi^i$ is a bounded linear operator 
from $L^2_{E^i} ({\mathcal X})$ to the finite-dimensional 
space $\mathcal{H}^i \subset C^\infty ({\mathcal X})$ 
we see that, for any $s'\in \mathbb N$, $0\leq \lambda'< 1$, 
\begin{equation*}
 \sup_{t \in [0,T]} \| \partial^j_t [\Pi^i u (\cdot, t) 
 \|_{C^{s',\lambda'} _{E^i}(\mathcal{X})} =    \sup_{t \in [0,T]}
   \| \Pi^i  \partial^j_t u(\cdot, t)  \|_{C^{s',\lambda'}_{E^i} (\mathcal{X})} \leq 
\end{equation*}
\begin{equation*}
c_1\, \sup_{t \in [0,T]} \| \Pi^i  \partial_t^j u (\cdot, t) \|_{L^{2}_{E^i} (\mathcal{X})}
\leq c_2\, \sup_{t \in [0,T]}
   \| \partial_t^j u (\cdot, t) \|_{L^{2}_{E^i} (\mathcal{X})} \leq 
\end{equation*}
\begin{equation*}
   c_3\, \sup_{t \in [0,T]}  \| \partial_t^j u (\cdot, t) \|_{C^{2(s-j)+k,\lambda} 
_{E^i}(\mathcal{X})}
\end{equation*}
for all $0 \leq j \leq s$ and  $u \in 
C^{2s+k,\lambda,s,\frac{\lambda}{2}}_{E^i} (\mathcal{X}_T)$ 
where the existence of the constant $c_1$ is granted by 
the well known property of the finite-dimensional spaces, 
the constant $c_2$ is granted by the continuity 
of the projection $\Pi^i$ on $L^2_{E^i} (\mathcal{X})$ 
and the constant $c_3$ is granted by the continuity of the 
embedding $C^{2(s-j)+k,\lambda}_{E^i} (\mathcal{X}) \to L^2 _{E^i}(\mathcal{X})$.

Similarly, 
for any $k'\in \mathbb N$, 
\begin{equation*}
\sup_{t',t'' \in [0,T] \atop t' \neq t''}\frac{ \| \Pi^i  \partial^j_t u\, (\cdot, t') - 
\Pi^i  \partial^j_t u\, (\cdot, t'') \|_{C^{2(s-j)+k',0}_{E^i} (\mathcal{X})}}
{|t'-t''|^{\frac{\lambda}{2}}} \leq 
\end{equation*}
\begin{equation*}
c_1\, \sup_{t',t'' \in [0,T] \atop t' \neq t''}\frac{ \| \Pi^i (\partial^j_t u (\cdot, t') - 
\partial^j_t u (\cdot, t'')) \|_{L^{2} _{E^i}(\mathcal{X})}}
        {|t'-t''|^{\frac{\lambda}{2}}} \leq
\end{equation*}
\begin{equation*}
   c_2\,
   \sup_{t',t'' \in [0,T] \atop t' \neq t''}
   \frac{ \| \partial^j_t u (\cdot, t') - \partial^j_t u (\cdot, t'')
          \|_{C^{2 (s-j)+k,0}_{E^i} (\mathcal{X})}}
        {|t'-t''|^{\frac{\lambda}{2}}} 
\end{equation*}
where the existence of the constant $c_1$ is granted by the well known property of the 
finite-dimensional spaces, the constant $c_2$ is granted by the continuity of the projection 
$\Pi^i$ and the continuity of the embedding $C^{2(s-j)+k',0}_{E^i} (\mathcal{X}) \to 
L^2_{E^i} (\mathcal{X})$.
\end{proof}

Now, if $u$ belongs to the kernel of operator \eqref{eq.Hoelder.d.t} then
\begin{equation} 
\label{eq.ker.d+d*}
A^i u (\cdot,t) =0 \mbox{ and } (A^{i-1})^* u (\cdot,t) =0 \mbox{ for each } 
t \in [0,T]. 
\end{equation}
This is equivalent to the fact that $\Pi^i u (\cdot,t) = u(\cdot,t)$
 for each $t \in [0,T]$.  
Hence the kernel of operator \eqref{eq.Hoelder.d.t} equals to $C^{s,\frac{\lambda}{2}} 
([0,T], \mathcal{H}^i)$.   

\begin{lemma}
\label{l.Phi.t}
The pseudo-differential operator $\varPhi_i $ induces continuous maps  
\begin{equation} \label{eq.Phi.d.loss}
\varPhi_i :\, C^{2s+k,\lambda,s,\frac{\lambda}{2}}_{E^{i+1}} (\mathcal{X}_T) 
\to C^{2s+k,\lambda,s,\frac{\lambda}{2}}_{E^i} (\mathcal{X}_T)
 \cap \mathcal{S}_{(A^{i-1})^*},
\end{equation}
\begin{equation} \label{eq.Phi.d}
\varPhi_i :\, C^{2s+k,\lambda,s,\frac{\lambda}{2}}_{E^{i+1}} (\mathcal{X}_T) \cap 
\mathcal{D}_{A^{i+1}} \to C^{2s+k,\lambda,s,\frac{\lambda}{2}}_{E^i} (\mathcal{X}_T)
\cap \mathcal{D}_{A^i} \cap \mathcal{S}_{(A^{i-1})^*},
\end{equation}
satisfying \eqref{eq.Hodge.1}, \eqref{eq.Hodge.2} on the space 
$C^{2s+k,\lambda,s,\frac{\lambda}{2}}_{E^i} (\mathcal{X}_T)
\cap \mathcal{D}_{A^i} $.  
\end{lemma}

\begin{proof} 
The use of Lemma \ref{l.Phi} yields
\begin{equation*}
   \sup_{t \in[0,T]}
   \| \partial^j_t \varPhi_i  f (\cdot, t) 
 \|_{C^{2(s-j)+k+1,\lambda}_{E^i} (\mathcal{X})}
 = \sup_{t \in [0,T]} \| \varPhi_i  \partial^j_t f (\cdot, t)  
     \|_{C^{2(s-j)+k+1,\lambda}_{E^i} (\mathcal{X})} \leq 
\end{equation*}   
\begin{equation*}
c\,   \sup_{t \in [0,T]}
   \| \partial_t^j f (\cdot, t) \|_{C^{2(s-j)+k,\lambda} _{E^i} (\mathcal{X})} 
\end{equation*}
for all $0 \leq j \leq s$ and  $f \in 
C^{2s+k, \lambda,s,\frac{\lambda}{2}}_{E^{i+1}} (\mathcal{X}_T)$. 
Besides, 
\begin{equation*}
 \sup_{t',t'' \in [0,T] \atop t' \neq t''} \frac{ \| \varPhi_i  \partial^j_t f\, 
(\cdot, t') - \varPhi_i  \partial^j_t f\, (\cdot, t'')
          \|_{C^{2 (s-j)+k,0}_{E^{i}} (\mathcal{X})}}
        {|t'-t''|^{\frac{\lambda}{2}}}
\leq
\end{equation*}
\begin{equation*} 
   c_1 \,\sup_{t',t'' \in [0,T] \atop t' \neq t''}
   \frac{ \| \partial^j_t f (\cdot, t') - \partial^j_t f (\cdot, t'')
          \|_{C^{2(s-j)+k-1,\lambda}_{E^{i+1}} (\mathcal{X})}}
        {|t'-t''|^{\frac{\lambda}{2}}}
 \leq 
\end{equation*}
\begin{equation*}
   c_2\,
   \sup_{t',t'' \in \mathcal{I}_T \atop t' \neq t''}
   \frac{ \| \partial^j_t f (\cdot, t') - \partial^j_t f (\cdot, t'')
          \|_{C^{2 (s-j)+k,0}_{E^{i+1}} (\mathcal{X})}}
        {|t'-t''|^{\frac{\lambda}{2}}}
\end{equation*}
for all $f \in   C^{2s+k,\lambda,s,\frac{\lambda}{2}}_{E^{i+1}} (\mathcal{X}_T)$,
 where  the constant $c_1$ is granted by the continuity 
of the operator $\varPhi_i$ on the scale $C^{s,\lambda} _{E^i} (\mathcal X)$ 
and $c_2$ is a constant granted by Embedding Theorem 
\ref{t.emb.hoelder.t}. Thus we conclude that 
  operator \eqref{eq.Phi.d.loss} is bounded. 

Identities \eqref{eq.Hodge.1}, \eqref{eq.Hodge.2}, are still valid 
for each $t \in [0,T]$ with $t$ regarded as a parameter if 
$f \in  C^{2s+k,\lambda,s,\frac{\lambda}{2}}_{E^{i+1}} (\mathcal{X}_T) \cap 
\mathcal{D}_{A^{i+1}} $.  
Then the definition of the operator $\varPhi_i$ and 
\eqref{eq.Hodge.2} imply that for each $f \in  C^{2s+k,\lambda,s,\frac{\lambda}{2}}
_{E^{i+1}} (\mathcal{X}_T) \cap \mathcal{D}_{A^{i+1}} $
\begin{equation*}
(A^{i-1})^* \varPhi_i f =0 , \,\, A^i \varPhi_i f =  f  - 
 \varPhi_{i+1} A^{i+1}f- \Pi^{i+1} f.
\end{equation*}
 Thus, $\varPhi_i $ maps 
$C^{2s+k,\lambda,s,\frac{\lambda}{2}}_{E^{i+1}} (\mathcal{X}_T) \cap 
\mathcal{D}_{A^{i+1}}$ continuously to  
$ C^{2s+k,\lambda,s,\frac{\lambda}{2}}_{E^{i}} (\mathcal{X}_T) \cap 
\mathcal{D}_{A^{i}}$  because of the continuity of operators \eqref{eq.Pi}, 
\eqref{eq.Phi.d.loss}.  
\end{proof}

\begin{lemma}
\label{l.Phi.t.*}
The pseudo-differential operator $\hat \varPhi_i $ induces continuous maps  
\begin{equation} \label{eq.Phi.d.loss.*}
\hat \varPhi_i :\, C^{2s+k,\lambda,s,\frac{\lambda}{2}}_{E^{i}} (\mathcal{X}_T) 
\to C^{2s+k,\lambda,s,\frac{\lambda}{2}}_{E^{i+1}} (\mathcal{X}_T)
 \cap \mathcal{S}_{A^{i+1}},
\end{equation}
\begin{equation} \label{eq.Phi.d.*}
\hat \varPhi_i :\, C^{2s+k,\lambda,s,\frac{\lambda}{2}}_{E^{i}} (\mathcal{X}_T) \cap 
\mathcal{D}_{(A^{i-1})^*} \to C^{2s+k,\lambda,s,\frac{\lambda}{2}}_{E^{i+1}} (\mathcal{X}_T)
\cap \mathcal{D}_{(A^i)^*} \cap \mathcal{S}_{A^{i+1}},
\end{equation}
satisfying \eqref{eq.Hodge.1}, \eqref{eq.Hodge.3} on the space 
$C^{2s+k,\lambda,s,\frac{\lambda}{2}}_{E^i} (\mathcal{X}_T)
\cap \mathcal{D}_{(A^{i-1})^*} $.  
\end{lemma}

\begin{proof} 
Similar to the proof of lemma \ref{l.Phi.t.*}.
\end{proof}

The statement on the range of the operator  
\eqref{eq.Hoelder.d.t} follow because 
formulas \eqref{eq.Hodge.2}, \eqref{eq.Hodge.3} are still valid on 
the spaces $C^{2s+k,\lambda,s,\frac{\lambda}{2}}_{E^i} (\mathcal{X}_T)
\cap \mathcal{D}_{A^i}$ and $C^{2s+k,\lambda,s,\frac{\lambda}{2}}_{E^i} (\mathcal{X}_T)
\cap \mathcal{D}_{(A^{i-1})^*}$ respectively.
\end{proof}

Now we are ready to define the Leray-Helmholtz type projection 
onto the spaces $C^{s,\lambda} _{E^i}({\mathcal X}) \cap 
\mathcal{S}_{(A^{i-1})^*}$ and 
$C^{2s+k,\lambda,s,\frac{\lambda}{2}} _{E^i}({\mathcal X}_T) \cap 
\mathcal{S}_{(A^{i-1})^*}$. 

\begin{lemma} 
\label{l.Helmholtz}
Let $s,k \in {\mathbb Z}_+$, $0<\lambda<1$.  
The pseudo-differential operator $\pi^i=(A^i)^*  A^i \varphi^{i} + \Pi^i$ 
on $\mathcal{X}$ induce continuous surjective maps
\begin{equation} \label{eq.proj.cont}
   \pi^i :\, C^{s,\lambda} _{E^{i}}(\mathcal{X}) \to
 C^{s,\lambda}_{E^i} (\mathcal{X}) \cap \mathcal{S}_{(A^{i-1})^*} ,
\end{equation}
\begin{equation} \label{eq.proj.cont.t}
\pi^i :\, C^{2s+k,\lambda,s,\frac{\lambda}{2}} _{E^{i}}(\mathcal{X}_T) 
\cap \mathcal{D}_{A^i}\to
 C^{2s+k,\lambda,s,\frac{\lambda}{2}}_{E^i} (\mathcal{X}_T)
\cap \mathcal{D}_{A^i}  \cap \mathcal{S}_{(A^{i-1})^*} ,
\end{equation}
\begin{equation} \label{eq.proj.cont.t.loss}
\pi^i :\, C^{2s+k+1,\lambda,s,\frac{\lambda}{2}} _{E^{i}}(\mathcal{X}_T) 
\to C^{2s+k,\lambda,s,\frac{\lambda}{2}}_{E^i} (\mathcal{X}_T) \cap \mathcal{D}_{A^i} 
\cap \mathcal{S}_{(A^{i-1})^*} , 
\end{equation}
satisfying
\begin{equation}\label{eq.proj.1}
\pi^{i} \circ \pi^i u =\pi^i u,\, (\pi^{i} u,u)_{i} = 
(u,\pi^{i} u)_{i},  \, 
 (\pi^{i} u, (I-\pi^i)u)_{i}  =0,
\end{equation}
for all $u\in C^{s,\lambda} _{E^{i}}(\mathcal{X})$ or $u \in 
C^{2s+k,\lambda,s,\frac{\lambda}{2}} _{E^{i}}(\mathcal{X}_T) \cap \mathcal{D}_{A^i}$. 
\end{lemma}

\begin{proof} Using \eqref{eq.Hodge.4}, \eqref{eq.Hodge.5} we see that 
\begin{equation} \label{eq.pi.Dom}
\pi^i = \varPhi _i A^i+\Pi^i \mbox{ on  } 
C^{2s+k,\lambda,s,\frac{\lambda}{2}} _{E^{i}}(\mathcal{X}_T) 
\cap \mathcal{D}_{A^i}.
\end{equation}
Therefore the continuity 
of operators \eqref{eq.proj.cont}, \eqref{eq.proj.cont.t},  \eqref{eq.proj.cont.t.loss} 
follows from Lemmata \ref{l.Phi}, \ref{l.Phi.t}, formula \eqref{eq.Hodge.0} and the 
property of complex \eqref{eq.ellcomp}: $(A^{i-1})^* \circ (A^{i})^* \equiv 0$. 

The surjectivity of the operators follows from formula \eqref{eq.Hodge.3}. 
Indeed, if $v\in  C^{s,\lambda}_{E^i} (\mathcal{X}) \cap \mathcal{S}_{(A^{i-1})^*}$ then
\eqref{eq.Hodge.3} and Lemma \ref{l.Phi} imply 
\begin{equation*}
v =(A^i)^*\hat \varPhi^i v +\Pi ^i v= (A^i)^*A^i \varphi^i v+\Pi ^i v.
\end{equation*} 
If $v\in  C^{2s+k,\lambda,s, \frac{\lambda}{2} }_{E^i} (\mathcal{X}_T) \cap 
\mathcal{D}_{A^i}  \cap \mathcal{S}
_{(A^{i-1})^*}$  
then, using \eqref{eq.Hodge.3} and   Corollary \ref{c.Hoelder.d.t}, 
we again conclude that 
\begin{equation*}
v = (A^i)^*A^i \varphi^i v+\Pi ^i v= \pi^i v.
\end{equation*} 

Next, using \eqref{eq.Hodge.0}, \eqref{eq.Hodge.4}, \eqref{eq.Hodge.5}, we see that
\begin{equation*}
\pi^{i} \circ \pi^i u = ((A^i)^*  A^i \varphi^{i} + \Pi^i)\circ
((A^i)^*  A^i  \varphi^{i}+ \Pi^i)u =
\end{equation*}
\begin{equation*}
((A^i)^* A^i (A^i)^* \varphi^{i+1}   A^i \varphi^{i} + \Pi^i) u = \pi^i u
\end{equation*}
for all $u\in C^{s,\lambda} _{E^{i}}(\mathcal{X})$ or $u \in 
C^{2s+k,\lambda,s,\frac{\lambda}{2}} _{E^{i}}(\mathcal{X}_T) \cap \mathcal{D}_{A^i}$.

Finally, we recall that $\pi^i u \in \mathcal{S}_{(A^{i-1})^*}$ and then, 
by \eqref{eq.Hodge.0}, 
\begin{equation*}
(\pi^{i} u, u)_{i} = 
( \pi^i u , \pi^i u + A^{i-1} (A^{i-1})^*\varphi ^{i})
_{i}=
\end{equation*}
\begin{equation*}
= (\pi^{i} u, \pi^i u)_{i} = 
(u,\pi^{i} u)_{i},
\end{equation*} 
\begin{equation*}
(\pi^{i} u, (I-\pi^i)u)_{i} = 
(\pi^{i} u,u)_{i} - (\pi^i u, \pi^i u)_{i}
=0
\end{equation*}
 for all $u\in C^{s,\lambda} _{E^{i}}(\mathcal{X})$ or $u \in 
C^{2s+k,\lambda,s,\frac{\lambda}{2}} _{E^{i}}(\mathcal{X}_T) \cap \mathcal{D}_{A^i}$. 
\end{proof}

We finish the section with two important lemmata specifying the action 
of the non-linear operator ${\mathcal N}^i$ and its linearisations.

\begin{lemma} \label{l.N.cont}
Let $s,k\in \mathbb N$,  $0<\lambda<1$.  
Bilinear forms \eqref{eq.nonlinear} induce continuous non-linear operators
\begin{equation} \label{eq.N.bound}
{\mathcal N}^i : 
C^{2s+k,\lambda,s,\frac{\lambda}{2}} _{E^i} ({\mathcal X}_T)\cap
\mathcal{D}_{A^{i}} \to 
C^{2s+k-1,\lambda,s,\frac{\lambda}{2}} _{E^i} ({\mathcal X}_T)\cap
\mathcal{D}_{A^{i}},
\end{equation}
\begin{equation} \label{eq.Ni.t}
\pi^i {\mathcal N}^i : C^{2s+k,\lambda,s,\frac{\lambda}{2}} _{E^i} ({\mathcal X}_T) \cap
\mathcal{D}_{A^{i}} \to 
C^{2s+k-1,\lambda,s,\frac{\lambda}{2}} _{E^i} ({\mathcal X}_T) \cap \mathcal{D}_{A^{i}}.
\end{equation}
\end{lemma}

\begin{proof} By the definition of $\pi^i$, we see that 
\begin{equation} \label{eq.orthog.range}
(A^{i-1}v, \pi^i u)_{i} = 
(v, (A^{i-1})^*\pi^i u)_{i-1} =0
\end{equation}
for all $v \in C^{1,\lambda,0,0} _{E^{i-1}} ({\mathcal X}_T)$, 
$u \in C^{1,\lambda,0,0} _{E^i} ({\mathcal X}_T)$, i.e. 
\begin{equation} \label{eq.pi.Ni}
\pi^i {\mathcal N}^i  (u ) = \pi^i
{\mathcal M}_{i,1} (A^{i} u,u) 
\end{equation}
for all $u\in C^{2s+k,\lambda,s,\frac{\lambda}{2}} _{E^i} ({\mathcal X}_T)$. 
Hence Lemma \ref{l.product} and formulas \eqref{eq.property.M1}, \eqref{eq.M.diff} 
imply that operators 
${\mathcal N}^i$ and $\pi^i {\mathcal N}^i$ map $C^{2s+k,\lambda,s,\frac{\lambda}{2}} 
_{E^i} ({\mathcal X}_T) \cap \mathcal{D}_{A^{i}}$ continuously to 
$C^{2s+k-1,\lambda,s,\frac{\lambda}{2}} _{E^i} ({\mathcal X}_T)$. 

Finally, as $A^{i} \circ A^{i-1} \equiv 0$, then, according to Lemmata \ref{l.product}, 
\ref{l.Phi.t} and \eqref{eq.pi.Dom},
\begin{equation*}
A^i {\mathcal N}^i  (u ) = A^i {\mathcal M}_{i,1} (A^{i} u,u) \in 
C^{2s+k-1,\lambda,s,\frac{\lambda}{2}} _{E^i} ({\mathcal X}_T)
\cap \mathcal{D}_{A^{i+1}}, 
\end{equation*}
\begin{equation*}
A^i \pi^i {\mathcal N}^i  (u ) = A^i \varPhi_i A^i {\mathcal M}_{i,1} (A^{i} u,u) \in 
C^{2s+k-1,\lambda,s,\frac{\lambda}{2}} _{E^i} ({\mathcal X}_T)
\end{equation*}
if $u \in C^{2s+k,\lambda,s,\frac{\lambda}{2}} 
_{E^i} ({\mathcal X}_T) \cap \mathcal{D}_{A^{i}}$.  
Thus, the continuity of operators \eqref{eq.N.bound}, \eqref{eq.Ni.t} follows again from 
Lemma \ref{l.product} and formulas \eqref{eq.property.M1}, \eqref{eq.M.diff}. 
\end{proof}

\begin{lemma} \label{l.N.der}
Let $s,k\in \mathbb N$,  $0<\lambda<1$. The Fr\'echet derivatives $({\mathcal N}^i)'
_{|u^{(0)}}$ and $(\pi^i {\mathcal N}^i)'_{|u^{(0)}}$ of continuous operators  
\eqref{eq.N.bound} and \eqref{eq.Ni.t} 
at the point $u^{(0)} \in C^{2s+k,\lambda,s,\frac{\lambda}{2}} _{E^i} ({\mathcal X}_T) 
\cap \mathcal{D}_{A^{i}} $  equal to 
\begin{equation*}
({\mathcal N}^i)'_{|u^{(0)}} v = 
 {\mathcal M}_{i,1} (A^{i}v,u^{(0)}) + {\mathcal M}_{i,1} (A^{i} u^{(0)},v) +
A^{i-1}  \Big( {\mathcal M}_{i,2} ( v,u^{(0)}) + {\mathcal M}_{i,2} ( u^{(0)},v)\Big), 
\end{equation*}
\begin{equation*}
(\pi^i {\mathcal N}^i)'_{|u^{(0)}} v = 
\pi^i \Big( {\mathcal M}_{i,1} (A^{i}v,u^{(1)}) + {\mathcal M}_{i,1} (A^{i} u^{(1)},v) \Big), 
\end{equation*}
respectively; these are the bounded linear operators
\begin{equation} \label{eq.der.N.Dom}
({\mathcal N}^i)'_{|u^{(0)}} : C^{2s+k,\lambda,s,\frac{\lambda}{2}} _{E^i} ({\mathcal X}_T) 
\cap \mathcal{D}_{A^{i}} \to 
C^{2s+k-1,\lambda,s,\frac{\lambda}{2}} _{E^i} ({\mathcal X}_T) 
\cap \mathcal{D}_{A^{i}},
\end{equation}
\begin{equation*} 
\pi ^i ({\mathcal N}^i)'_{|u^{(0)}} : 
C^{2s+k,\lambda,s,\frac{\lambda}{2}} _{E^i} ({\mathcal X}_T) \cap
\mathcal{D}_{A^{i}} \to 
C^{2s+k-1,\lambda,s,\frac{\lambda}{2}} _{E^i} ({\mathcal X}_T) \cap \mathcal{D}_{A^{i}}. 
\end{equation*}
\end{lemma}

\begin{proof} It follows from \eqref{eq.M.diff} that 
\begin{equation*}
{\mathcal N}^i  (u ) - {\mathcal N}^i (u^{(0)}) = 
\end{equation*} 
\begin{equation*}
{\mathcal M}_{i,1} (A^{i} (u-u^{(0)}),u^{(0)}) + 
{\mathcal M}_{i,1} (A^{i} u^{(0)},u-u^{(0)}) + {\mathcal M}_{i,1} (A^{i} 
(u-u^{(0)}),u-u^{(0)})+
\end{equation*} 
\begin{equation*}
A^{i-1}  \Big( {\mathcal M}_{i,2} ( u-u^{(0)},u^{(0)}) + 
{\mathcal M}_{i,2} ( u^{(0)},u-u^{(0)}) + {\mathcal M}_{i,2} (u-u^{(0)},u-u^{(0)})\Big)
\end{equation*} 
if $u,u^{(0)} \in C^{2s+k,\lambda,s,\frac{\lambda}{2}} _{E^i} ({\mathcal X}_T) \cap
 \mathcal{D}_{A^{i}}$. 
Then the statement follows from Lemmata \ref{l.product}, \ref{l.Phi.t} and \ref{l.N.cont} 
because ${\mathcal M}_{i,j}$ are bilinear forms, 
\begin{equation*}
A^i \Big( {\mathcal N}^i  (u ) - {\mathcal N}^i (u^{(0)}) \Big)= 
\end{equation*}
\begin{equation*}
A^i \Big( {\mathcal M}_{i,1} (A^{i} (u-u^{(0)}),u^{(0)}) + 
{\mathcal M}_{i,1} (A^{i} u^{(0)},u-u^{(0)}) + {\mathcal M}_{i,1} (A^{i} 
(u-u^{(0)}),u-u^{(0)})\Big)
\end{equation*} 
and, according to formula \eqref{eq.pi.Dom} and Lemmata \ref{l.product}, \ref{l.Phi.t}, 
\begin{equation*}
A^i \Big( {\mathcal M}_{i,1} (A^{i}v,u^{(0)}) + {\mathcal M}_{i,1} (A^{i} u^{(0)},v) \Big)
\in C^{2s+k-1,\lambda,s,\frac{\lambda}{2}} _{E^i} ({\mathcal X}_T) 
\cap \mathcal{D}_{A^{i+1}}
\end{equation*}
\begin{equation*}
A^i \pi^i\Big( {\mathcal M}_{i,1} (A^{i}v,u^{(0)}) + {\mathcal M}_{i,1} (A^{i} u^{(0)},v) \Big)
=\end{equation*}
\begin{equation*}
A^i \varPhi_i 
A^i\Big( {\mathcal M}_{i,1} (A^{i}v,u^{(0)}) + {\mathcal M}_{i,1} (A^{i} u^{(0)},v) \Big)
\in C^{2s+k-1,\lambda,s,\frac{\lambda}{2}} _{E^i} ({\mathcal X}_T) 
\cap \mathcal{D}_{A^{i+1}}
\end{equation*}
if $v,u,u^{(0)} \in C^{2s+k,\lambda,s,\frac{\lambda}{2}} _{E^i} ({\mathcal X}_T) \cap
 \mathcal{D}_{A^{i}}$.   
\end{proof}

\section{Parabolic operators over the H\"older spaces}
\label{s.thoitwHs}

As usual, we denote by $\gamma_{t_0} u$ the restriction of a continuous function $u$ in the 
layer $\mathcal{X}_T $ to the set $\{ t = t_0 \}$ in $\mathcal{X}_T$, where
   $t_0 \in [0,T]$. The following lemma is obvious.

\begin{lemma}
\label{l.bound.trace.holder}
Let $s, k \in \mathbb{Z}_{+}$ and $\lambda \in [0,1)$.
The restriction $\gamma_0$ induces a bounded linear operator
$   \gamma_0 :\, C^{2s+k,\lambda,s,\frac{\lambda}{2}}_{E^i} (\mathcal{X}_T)
            \to C^{2s+k,\lambda}_{E^i} (\mathcal{X})$.
\end{lemma}

Consider the following Cauchy 
problem: given a section  $u _0$ of the bundle $E^i$ over $\mathcal{X}$,
find a section $u$ of the induced bundle $E^i (t)$ over $\mathcal{X}_T $, such that
\begin{equation}
\label{pr.Cauchy.heat.homo}
\left\{ \begin{array}{rclll}
        L_\mu ^i u (x,t)
        & =
        & 0 
        & \mbox{for}
        & (x,t) \in \mathcal{X} \times (0,T),
\\
          (\gamma_0 u) (x)
        & =
        & u_0 (x)
        & \mbox{for}
        & x \in \mathcal{X},
        \end{array}
\right.
\end{equation}
where $L^i_{\mu} = \partial_ t + \mu \Delta ^i$ 
is an evolution operator on semiaxis $t>0$ with  $\mu>0$. 

If $\mu >0$ then it is easily seen that the endomorphism of $E^i_x$ given by the 
principal symbol
   $- \mu \left( \sigma^2 (\Delta^i) (x,\xi) \right)$
of the operator $L^i_{\mu}$ at any nonzero vector
$\xi $ of the cotangent bundle $T^\ast_x \mathcal{X}$ has only real eigenvalues strictly less 
than zero. Therefore, the operator $L^i_{\mu}$  
is actually one of the well-known type of
partial differential operators, called parabolic, which enjoy a behaviour essentially
like that of the classical equation of the heat conduction.
For instance, one has as a smoothing operator the fundamental solution
$   t \mapsto \psi^i_{\mu} (x,y,t) \in C^{\infty} (\mathcal{X} \times \mathcal{X}, 
E^i \otimes E^i{}^\ast)$
which generates the solution of the Cauchy problem \eqref{pr.Cauchy.heat.homo} 
for any $u_0 \in C^{\infty}_{E^i} (\mathcal{X})$,
\begin{equation*}
   (\varPsi^{(i,in)} _{\mu} u_0) (x,t)
 = \int_{\mathcal{X}} (u_0, \ast^{-1} \psi^i_{\mu} (x,\cdot,t))_{y} dy.
\end{equation*}
The fundamental solution $\varPsi^i_{\mu}$ is here unique by virtue of 
the compactness of $\mathcal{X}$.

We recall briefly the Hilbert-Levi procedure for constructing the fundamental solutions
for parabolic operators, and state certain basic estimate for them, only to the extent
which we shall need later.
For the details we refer the reader to \cite{Eide56}, \cite{Eide90} and elsewhere.

We can cover $\mathcal{X}$ by a finite number of coordinate patches $U$ such that over each $U$
the bundle $E^i$ is trivial, i.e., the restriction $E^i \restriction_U$ is isomorphic to
the trivial bundle $U \times \mathbb{C}^{k_i}$, where $k_i$ is the fibre dimension of
$E^i$. Let $x = (x^1, \ldots, x^n)$ be local coordinates in $U$.
Then any section of $E^i$ over $U$ can be regarded as a $k_i$-column of 
complex-valued functions of coordinates $x$.
Under this identification, for any $u \in C^\infty_{E^i} (U)$, 
the $k_i\,$-column $\Delta^i u$
of functions in $U$ is represented by
\begin{equation*}
   \Delta^i  u
 =  \sum_{|\alpha| \leq 2} \varDelta_{\alpha}^i (x)\, \partial^{\alpha}  u,
\end{equation*}
where
   $\Delta_{\alpha}^i (x)$ are $(k_i \times k_i)\,$-matrices of $C^\infty$ functions of $x$,
   by $\partial^{\alpha}$ is meant the derivative
   $(\partial/\partial x^1)^{\alpha_1} \ldots (\partial/\partial x^n)^{\alpha_n}$.
So, the principal symbol of $\Delta^i$ is given by
\begin{equation*}
   \sigma^{2} (\Delta^i) (x,\xi)
 = - \sum_{|\alpha| = 2} \Delta_{\alpha}^i (x)\, \xi^{\alpha} 
\end{equation*}
for $(x,\xi) \in U \times {\mathbb R}^n$, where
   $\xi^{\alpha} = \xi_1^{\alpha_1} \ldots \xi_n^{\alpha_n}$.
As already mentioned, the map $\sigma^{2} (\Delta^i) (x,\xi)$ has only strictly 
negative eigenvalues for each nonzero vector $\xi$ and hence the following matrix 
is well defined for all $x, y \in U$ and $t > 0$:
\begin{equation}
\label{eq.paramet}
   P^i_{U} (x,y,t) = \frac{1}{(2 \pi)^n}
  \int \exp \left( t\, \mu \sigma^{2} (\Delta^i) (x,\xi) + 
\sqrt{-1}\, \langle x-y, \xi \rangle\, I \right) d \xi;
\end{equation}
here $I$ stands for the unity matrix of type $k_i \times k_i$.

It will be seen below that $p^i_U$ describes the principal part of singularity for the fundamental
solution of the parabolic operator $L^i_{\mu}$ in $U$.
For that reason $p^i_U$ is said to be a local parametrix for this operator.
Letting now
\begin{equation*}
   \sum \phi_U^2 \equiv 1
\end{equation*}
be a quadratic partition of unity on $\mathcal{X}$ subordinate to a covering $\{\, U\, \}$, 
and considering the expression $\phi_U (x) P^i_{U} (x,y,t) \phi_U (y)$
as a section of $E^i \otimes (E^i)^\ast$ with support in $U \times U$, we define a global 
parametrix for $L^i_\mu$  by
\begin{equation*}
   P^i (x,y,t) = \sum_U \phi_U (x) P^i_U (x,y,t) \phi_U (y).
\end{equation*}
The fundamental solution $\psi^i_{\mu} (x,y,t)$ of $L^i_\mu$ will then be given as the unique 
solution of the integral equation of Volterra type
\begin{equation*}
   \psi^i_{\mu} (x,y,t)
 = P^i (x,y,t)
 - \int_0^t dt' \int_{\mathcal{X}} (L^i_{\mu} P^i(\cdot,y,t'), \ast^{-1} \psi_{\mu} 
(x,\cdot,t-t'))_{z} dz,
\end{equation*}
the operator $L^i_{\mu}$ acting in $z$.
Since the singularities of $P^i$ and $L^i_{\mu} P^i$ are 
relatively weak, we can solve 
the integral equation by the standard method of successive approximation.
The kernel $\psi_{\mu}^i (x,y,t)$ obtained in this way is easily verified to be of 
class $C^\infty$, when $t > 0$, and,
\begin{equation*}
   u (x,t)
 = \int_{\mathcal{X}} (u_0, \ast^{-1} \psi^i _{\mu} (x,\cdot,t))_{y} dy
\end{equation*}
satisfies
   $L^i_{\mu}u = 0$ for $t > 0$
and
   $u (x,0) = u_0 (x)$ for all $x \in \mathcal{X}$.

As for the estimates, note that if we define, relative to a Riemannian metric $ds$ on 
$\mathcal{X}$, a distance function to be
\begin{equation*}
   d (x,y) = \inf \int_{\widehat{xy}} ds
\end{equation*}
for $x, y \in \mathcal{X}$, where $\widehat{xy}$ is a path from $x$ to $y$, then
\begin{equation}
\label{eq.eotfsps}
   |\partial_x^{\alpha} \partial_y^{\beta} \psi_{\mu} (x,y,t)|
 \leq
   c\,
   \frac{1}{t^{\frac{\scriptstyle{n + |\alpha| + |\beta|}}{\scriptstyle{2}}}}\,
   \exp \Big( - c' \frac{(d (x,y))^{2}}{t}\Big)
\end{equation}
locally on $\mathcal{X}'$ and for each entry of the matrix.
Also, as can be inferred and in fact easily proved from \eqref{eq.paramet} and 
\eqref{eq.eotfsps},
we get
\begin{equation*}
   |\psi^i_{\mu} (x,y,t) - P^i_U (x,y,t)|
 \leq
   c\,
   \frac{1}{t^{\frac{\scriptstyle{n-1}}{\scriptstyle{2}}}}\,
   \exp \Big( - c' \frac{(d (x,y))^{2}}{t}\Big)
\end{equation*}
uniformly on each compact subset of $U \times U$, showing that locally $p_U$ yields a first
approximation to the fundamental solution $\psi^i_{\mu} (x,y,t)$.

\begin{lemma}
\label{l.Greenhe}
Suppose that $\mu>0$ and 
$u \in C^{\infty}_{E^i} (\mathcal{X}_T)$. 
Then, for all $(x,t) \in \mathcal{X}_T$, 
\begin{equation}
\label{eq.Greenhe}
   u (x,t)
 = (\varPsi^{(i,in)}_{\mu} u (\cdot,0)) (x,t)
 + \int_0^t dt' \int_{\mathcal{X}} (L^i_{\mu}u (\cdot,t'), \ast^{-1} 
\psi_{\mu}^i (x,\cdot,t-t'))_{y} dy.
\end{equation}
\end{lemma}

\begin{proof}
As is well known, the Cauchy problem
\begin{equation}
\label{pr.Cauchy.heat}
\left\{ \begin{array}{rclll}
          L^i_{\mu}\, u (x,t)
        & =
        & f (x,t)
        & \mbox{for}
        & (x,t) \in \mathcal{X} \times (0,T),
\\
          (\gamma_0 u) (x)
        & =
        & u_0 (x)
        & \mbox{for}
        & x \in \mathcal{X}.
        \end{array}
\right.
\end{equation}
has a unique solution in $C^\infty_{E^i} (\mathcal{X}_T)$ for all smooth data $f$
in $\mathcal{X}_T$ and $u_0$ on $\mathcal{X}$.
By the above, the first term on the right-hand side of \eqref{eq.Greenhe} is a solution
of the Cauchy problem with $f = 0$ and $u_0 = u (\cdot,0)$.
Hence, we shall have established the lemma if we prove that the second term is a solution
of the Cauchy problem with $f = L_{\mu}u$ and $u_0 = 0$.
To this end, we rewrite the second summand on the right-hand side of \eqref{eq.Greenhe} in
the form
\begin{equation*}
   \int_0^t \varPsi^{(i,in)}_{\mu} 	\left( L^i_{\mu}u (\cdot,t') \right) (x,t-t')\, dt'
\end{equation*}
for $(x,t) \in \mathcal{X}_T$.
Obviously, the initial value at $t = 0$ of this integral vanishes, and so it remains to
calculate its image by $L_{\mu}$. Thus, we get 
\begin{equation*}
 L^i_{\mu} \int_0^t \varPsi^{(i,in)}_\mu \left( L^i_{\mu}u (\cdot,t') \right) (x,t-t')\, dt' = 
\end{equation*}
\begin{equation*}
   \varPsi_\mu^i \left( L^i_{\mu}u (\cdot,t) \right) (x,0)
 + \int_0^t L^i_{\mu}\, \varPsi^i_\mu \left( L_\mu ^i u (\cdot,t') \right) (x,t-t')\, dt'
 =    L_{\mu}^i u (x,t).
\end{equation*}
as desired.  
\end{proof}

For  $f \in C^{\infty}_{E^i} (\mathcal{X}_T)$, 
$u_0 \in C^{\infty} _{E^i} (\mathcal{X})$ we set
\begin{equation*}
   (\varPsi^{(i,v)} _{\mu} f) (x,t) = 
	 \int_0^t \varPsi^{(i,in)}_{\mu} 
	\left( f (\cdot,t') \right) (x,t-t')\, dt',
\end{equation*}
\begin{equation*}
   (\varPsi^{(i)} _{\mu} (f,u_0)) (x,t) = 
	 (\varPsi^{(i,v)} _{\mu} f) (x,t)  
	+(\varPsi^{(i,in)} _{\mu} u_0) (x,t).  
\end{equation*}
We can extend the action of the operator $\varPsi_
{\mu} ^i$ to the scale $C^{2s+k,\lambda,s, \frac{\lambda}{2}}_{E^i}  (\mathcal{X}_T)$. 

\begin{lemma}
\label{l.unique.heat.initial.hoelder} 
Problem \eqref{pr.Cauchy.heat} has at most one solution in  
$C^{2,\lambda,1,\frac{\lambda}{2}}_{E^i} (\mathcal{X}_T)$.
\end{lemma}

\begin{proof}
Cf. Theorem 16 in \cite[Ch.~1, \S~9]{Fri64}. 
\end{proof}

The solution of the Cauchy problem in  H\"{o}lder spaces is recovered from the data
by means of the Green formula.

\begin{lemma}
\label{l.bound.heat.initial.hoelder}
Assume that   $\mu > 0$.
Then,  for each function $u \in C^{2,\lambda,1,\frac{\lambda}{2}}_{E^i} 
(\mathcal{X}_T)$, it follows that $ u = \varPsi^i_{\mu}\, (L^i_ {\mu} u, \gamma_0 u) $.
\end{lemma}

\begin{proof}
See \em{ibid}.
\end{proof}

\begin{lemma}
\label{l.heat.key1}
Let $\mu>0$, $s$ be a positive integer,
   $k \in \mathbb{Z}_{+}$ and
   $0 < \lambda< 1$.
The parabolic potentials
   $\varPsi_{\mu}^{(i,v)}$ and
   $\varPsi^{(i,in)}_{\mu}$
induce bounded linear operators
\begin{equation*}
\begin{array}{rrcl}
   \varPsi^{(i,v)}_{\mu} :
 & C^{2(s-1)+k,\lambda,s-1,\frac{\lambda}{2}}_{E^i} (\mathcal{X}_T)
 & \to
 & C^{2s+k,\lambda,s,\frac{\lambda}{2}}_{E^i} (\mathcal{X}_T),
\\
   \varPsi^{(i,in)}_{\mu} :
 & C^{2 s+k,\lambda}_{E^i} (\mathcal{X})
 & \to
 & C^{2s+k,\lambda,s,\frac{\lambda}{2}}_{E^i} (\mathcal{X}_T) \cap \mathcal{S}_{L^i_{\mu}}.
\end{array}
\end{equation*}
\end{lemma}

\begin{proof} See \em{ibid}.
\end{proof}

The next result describes  the crucial property of 
the volume parabolic potential $\varPsi_{\mu}^i$
which we need in the sequel.

\begin{theorem}
\label{t.heat.key2}
Let $\mu>0$, $s\in \mathbb N$, $k \in \mathbb{Z}_{+}$ and 
$0 < \lambda < 1$. For any data $f  \in C^{2(s-1)+k,\lambda,s-1, \frac{\lambda}{2}} _{E^i}
(\mathcal{X}_T)$ and $u_0 \in C^{2s+k,\lambda}_{E^i} (\mathcal{X})$ the 
potential $\varPsi_{\mu}^i (f,u_0)$ is the unique solution 
to problem \eqref{pr.Cauchy.heat} in the space
$C^{2s+k,\lambda,s,\frac{\lambda}{2}}_{E^i} (\mathcal{X}_T)$.
\end{theorem}

\begin{proof} See \em{ibid}.
\end{proof}

\section{The linearised Navier-Stokes type equations}
\label{s.NS.lin}

Now we begin to study the operators related to a linearisation of the Navier-Stokes equations.
For this purpose, consider the following linear initial problem: 
\begin{equation}
\label{eq.NS.complex.lin}
\left\{
\begin{array}{lcl}
   \partial_t u + \mu \Delta^i u  + {V} ^i_0 u + A^{i-1} p
  =
  f & \mbox{ in } & \mathcal X\times (0,T), \\
   (A^{i-1})^*\, u  =0,\,\, (A^{i-2})^*\, p  =0   & \mbox{ in }  & \mathcal X\times [0,T], \\
u (x,0)= u_0 & \mbox{ in }  & \mathcal X,\\
 \end{array}
\right. 
\end{equation}
where (cf. Lemma \ref{l.N.der}) 
\begin{equation*}
   V_0^i u
 = {\mathcal M}_{i,1} (A^i u^{(0)}, u)+ {\mathcal M}_{i,1} (A^{i} u, u^{(0)})
 + (A^{i-1})^*  \Big( {\mathcal M}_{i,2} (u,u^{(0)}) + {\mathcal M}_{i,2} (u^{(0)},u)\Big)
\end{equation*}
is the linear first order term in $\mathcal{X}_T $ with a fixed sections $u^{(0)}$ 
of the bundle $E^i(t)$.  

\begin{theorem}
\label{t.NS.deriv.unique} 
Suppose that \eqref{eq.property.M1} holds. 
If $u^{(0)}\in  C^{0,0,0,0} _{E^i}(\mathcal{X}_T) \cap \mathcal{D}_{A^{i}}$ 
then for any pair $(u,p)$ satisfying \eqref{eq.NS.complex.lin} with $f=0$ and $u_0=0$, 
\begin{equation}
\label{eq.unicla}
 u\in
  C^{2,0,1,0}_{E^{i}} (\mathcal{X}_T ) \cap \mathcal{S}_{(A^{i-1})^*},\,\, 
   p  \in C^{0,0,0,0}_{E^{i-1}} (\mathcal{X}_T ) \cap \mathcal{D}_{A^{i-1}} , 
\end{equation}
the sections $u$ 
and $A^{i-1} p$ are identically zero.
\end{theorem}

\begin{proof} 
One may follow  \cite{Lera34a} or the proofs of 
Theorem 3.4 for $n = 3$ of \cite{Tema79},  Theorem 
\ref{t.NS.unique} below, proving the uniqueness result with the use of integration by parts.
\end{proof}

\begin{lemma}
\label{l.map.V0.bound}
Let $s,k\in \mathbb N$ and
   $0 < \lambda <  1$.
If $u^{(0)} \in C^{2s+k,\lambda,s,\frac{\lambda}{2}}_{E^i}({\mathcal X}_T)\cap
\mathcal{D}_{A^{i}}$ then
the 
Leray-Helmholtz projection $\pi^i$, the Fr\'echet derivative $({\mathcal N}^i)'_{|u^{(0)} }$ 
and the fundamental solution $\varPsi_\mu^i$ induce linear bounded operators
\begin{equation} \label{eq.Psipi.prime.bound.v}
\varPsi_\mu^{(i,v)} \pi^i: C^{2(s-1)+k,\lambda,s-1,\frac{\lambda}{2}} _{E^i} ({\mathcal X}_T)
\cap\mathcal{D}_{A^{i}} 
\to  C^{2s+k,\lambda,s,\frac{\lambda}{2}} _{E^i} ({\mathcal X}_T)\cap
\mathcal{D}_{A^{i}} \cap \mathcal{S}_{(A^{i-1})^*}, 
\end{equation}
\begin{equation} \label{eq.Psipi.prime.bound.in}
\varPsi_\mu^{(i,in)} : 
C^{2s+k+1,\lambda} _{E^i} ({\mathcal X}) \cap \mathcal{S}_{(A^{i-1})^*}
\to  C^{2s+k,\lambda,s,\frac{\lambda}{2}} _{E^i} ({\mathcal X}_T)\cap
\mathcal{D}_{A^{i}} \cap \mathcal{S}_{(A^{i-1})^*}, 
\end{equation}
\begin{equation} \label{eq.PsiN.prime.bound}
\varPsi_\mu^{(i,v)} \pi^i ({\mathcal N}^i)'_{|u^{(0)} }: 
C^{2s+k,\lambda,s,\frac{\lambda}{2}} _{E^i} ({\mathcal X}_T)\cap
\mathcal{D}_{A^{i}} \to 
C^{2(s+1)+k-1,\lambda,s+1,\frac{\lambda}{2}} _{E^i} ({\mathcal X}_T)\cap
\mathcal{D}_{A^{i}} \cap \mathcal{S}_{(A^{i-1})^*}.
\end{equation}
\end{lemma}

\begin{proof} According to Lemmata \ref{l.Helmholtz}, \ref{l.N.der}, \ref{l.heat.key1},  
and the Embedding Theorem \ref{t.emb.hoelder.t}, 
for each $u^{(0)} \in C^{2s+k,\lambda,s,\frac{\lambda}{2}} _{E^i} ({\mathcal X}_T)\cap
\mathcal{D}_{A^{i}} $ we have bounded linear operators:
\begin{equation*} 
\varPsi_\mu^{(i,v)} \pi^i: C^{2(s-1)+k,\lambda,s-1,\frac{\lambda}{2}} _{E^i} ({\mathcal X}_T)
\cap\mathcal{D}_{A^{i}} 
\to  C^{2s+k,\lambda,s,\frac{\lambda}{2}} _{E^i} ({\mathcal X}_T)\cap
\mathcal{D}_{A^{i}} , 
\end{equation*}
\begin{equation*} 
\varPsi_\mu^{(i,in)} : 
C^{2s+k+1,\lambda} _{E^i} ({\mathcal X}) 
\to  C^{2s+k,\lambda,s,\frac{\lambda}{2}} _{E^i} ({\mathcal X}_T)\cap
\mathcal{D}_{A^{i}} , 
\end{equation*}
\begin{equation*} 
\varPsi_\mu^{(i,v)} \pi^i ({\mathcal N}^i)'_{|u^{(0)} }: 
C^{2s+k,\lambda,s,\frac{\lambda}{2}} _{E^i} ({\mathcal X}_T)\cap
\mathcal{D}_{A^{i}} \to 
C^{2(s+1)+k-1,\lambda,s+1,\frac{\lambda}{2}} _{E^i} ({\mathcal X}_T)\cap
\mathcal{D}_{A^{i}} .
\end{equation*} 
On the other hand, 
\begin{equation*} 
(A^{i-1})^* L^i_\mu = L^i_\mu (A^{i-1})^* ,\, \Pi^i L^i_\mu =  L^i_\mu \Pi^i, \, 
\Phi_i L^{i+1}_\mu = L^i_\mu \Phi_i 
\end{equation*} 
\begin{equation*} 
(A^{i-1})^* \gamma_0 = \gamma_0 (A^{i-1})^*, \Pi^i \gamma_0 =  \gamma_0 \Pi^i, \,
\Phi_i  \gamma_0 = \gamma_0 \Phi_i 
\end{equation*} 
 by the very construction and then 
\begin{equation*} 
\pi^i L^i_\mu = L^i_\mu \pi^i ,\,\, 
\pi^i \gamma_0 = \gamma_0 \pi^i.
\end{equation*}  
Hence $\pi^i \varPsi_\mu^{(i)} =\varPsi_\mu^{(i)}  \pi^i$ 
because $\varPsi_\mu^{(i)} $ represents the inverse operator for $(L_\mu ^i, \gamma_0)$
on the scale $C^{2s+k,\lambda,s,\frac{\lambda}{2}} _{E^i} ({\mathcal X}_T)\cap
\mathcal{D}_{A^{i}}$. In particular, 
\begin{equation*} 
\pi^i \varPsi_\mu^{(i,v)} =\varPsi_\mu^{(i,v)}  \pi^i ,\,\, 
\pi^i \varPsi_\mu^{(i,v)} =\varPsi_\mu^{(i,v)}  \pi^i 
\end{equation*} 
that proves the continuity of the operators 
\eqref{eq.Psipi.prime.bound.v}, \eqref{eq.Psipi.prime.bound.in}  and 
\eqref{eq.PsiN.prime.bound}. 
\end{proof} 

According to Uniqueness Theorem \ref{t.NS.deriv.unique}, the `pressure' $p$ is defined 
by \eqref{eq.NS.complex.lin} up to the finitely-dimensional space 
$\mathcal{H}^{i-1}$. In order to achieve the uniqueness we will look for $p$ being   
$L^2_{E^{i-1}}(\mathcal{X})$-orthogonal to $\mathcal{H}^{i-1}$, i.e. 
\begin{equation}
\label{eq.p.ort}
\Pi^{i-1} p (\cdot,t) =0 \mbox{ for each } t \in [0,T].
\end{equation} 

Now, let $C^{2s+k,\lambda,s\frac{\lambda}{2}} _{E^{i-1}} (\mathcal{X}_T) \ominus 
\mathcal{H}^{i-1}$ stand for all sections from % the %Banach space 
$C^{2s+k,\lambda,s\frac{\lambda}{2}} _{E^ {i-1}} (\mathcal{X}_T)$,
satisfying \eqref{eq.p.ort}. Then we introduce 
%denote by ${\mathcal B}^{k,s,\lambda}_{i,1}$ and ${\mathcal B}^{k,s,\lambda}_{i,2}$ 
the following  Banach spaces: 
${\mathcal B}^{k,s,\lambda}_{i,1} =$
\begin{equation} \label{eq.sol.smoothness}
    C^{2s+k,\lambda,s,\frac{\lambda}{2}} _{E^i} 
({\mathcal X}_T)\cap\mathcal{D}_{A^{i}} \cap \mathcal{S}_{(A^{i-1})^*}  \times  
C^{2(s-1)+k,\lambda,s-1,\frac{\lambda}{2}} _{E^{i-1}} ({\mathcal X}_T)\cap
\mathcal{D}_{A^{i-1}} \cap \mathcal{S}_{(A^{i-2})*}
\ominus C^{s,\frac{\lambda}{2}}([0,T],\mathcal{H}^{i-1}),
\end{equation}
\begin{equation} \label{eq.data.smoothness}
{\mathcal B}^{k,s,\lambda}_{i,2} =
 C^{2(s-1)+k,\lambda,s-1,\frac{\lambda}{2}} _{E^i} ({\mathcal X}_T)\cap
\mathcal{D}_{A^{i}}
 \times
   C^{2s+k+1,\lambda} _{E^i} ({\mathcal X})\cap\mathcal{S}_{(A^{i-1})^*}
\end{equation}
The following lemma is just an adaptation to the particular function spaces of 
the standard reduction of the linearized Navier-Stokes types 
equation to a pseudo-differential equation that does not involve 
the `pressure' $p$. 

\begin{lemma}
\label{l.reduce.psi}
Assume that   $s,k \in \mathbb N$ and
    $u^{(0)}\in C^{2s+k,\lambda,s,\frac{\lambda}{2}} _{E^i} ({\mathcal X}_T)\cap
\mathcal{D}_{A^{i}}$.     
Let moreover $(f,u_0)$ be an arbitrary pair of ${\mathcal B}^{k,s,\lambda}_{i,2}$. 
Then there is a solution $(u,p)$ of class ${\mathcal B}^{k,s,\lambda}_{i,1}$ 
 to problem \eqref{eq.NS.complex.lin} if and only if there is a solution 
$v \in  C^{2s+k,\lambda,s,\frac{\lambda}{2}} _{E^i} ({\mathcal X}_T) \cap
\mathcal{D}_{A^{i}} \cap \mathcal{S}_{(A^{i-1})^*}$ to pseudo-differential equation 
\begin{equation}
\label{eq.NS.lin.psi}
  v + \varPsi_\mu^{(i,v)} \pi^i ({\mathcal N}^i)'_{|u^{(0)} } v = F
\end{equation}
with the datum $F=\varPsi_\mu^i (\pi^i f, u_0) \in 
C^{2s+k,\lambda,s,\frac{\lambda}{2}} _{E^i} ({\mathcal X}_T)\cap
\mathcal{D}_{A^{i}} \cap \mathcal{S}_{(A^{i-1})^*}$. Besides,  
\begin{equation} 
\label{eq.p.lin}
 u=v, \, \,  p = \varPhi_{i-1}  (I-\pi^i) \Big(f - ({\mathcal N}^i)'_{|u^{(0)} } v  \Big).
\end{equation}
\end{lemma}

\begin{proof} First of all, we note that $\pi^i V^i_0 = \pi^i ({\mathcal N}^i)'_{|u^{(0)} }$ 
in this case, see Lemma \ref{l.N.der}. 

If $(u,p)$ is a solution to problem \eqref{eq.NS.complex.lin} from class 
\eqref{eq.sol.smoothness} then, applying the bounded 
linear operator $\varPsi_\mu^{i} \pi^i$ to  \eqref{eq.NS.complex.lin}
we see that  the section $u=v$ is a solution to \eqref{eq.NS.lin.psi} 
with the datum $F=\varPsi_\mu^i (\pi^i f,  u_0) \in 
C^{2s+k,\lambda,s,\frac{\lambda}{2}} _{E^i} ({\mathcal X}_T)\cap
\mathcal{D}_{A^{i}} \cap \mathcal{S}_{(A^{i-1})^*}$. 

Let $v \in  C^{2s+k,\lambda,s,\frac{\lambda}{2}} _{E^i} ({\mathcal X}_T) \cap
\mathcal{D}_{A^{i}} \cap \mathcal{S}_{(A^{i-1})^*}$ be a solution to \eqref{eq.NS.lin.psi} with the datum $F=\varPsi_\mu^i  (\pi^i f,  u_0) \in 
C^{2s+k,\lambda,s,\frac{\lambda}{2}} _{E^i} ({\mathcal X}_T)\cap
\mathcal{D}_{A^{i}} \cap \mathcal{S}_{(A^{i-1})^*}$. Then by Lemma 
\ref{l.bound.heat.initial.hoelder}
we conclude that 
\begin{equation} \label{eq.sol.v}
   \left\{
   \begin{array}{rclll}
     L^i_{\mu} v +  \pi^i ({\mathcal N}^i)'_{|u^{(0)} } v  
   & =
   & \pi^i f
   & \mbox{in}
   & \mathcal{X} \times (0,T) ,\\
	(A^{i-1})^*v & = & 0 & \mbox{in}
   & \mathcal{X} \times [0,T] ,
\\
     \gamma_0\, v
   & =
   &  u_0
   & \mbox{on}
   & \mathcal{X}.
\end{array}
\right.
\end{equation}

Set $(u,p)$ as in \eqref{eq.p.lin}. As $(A^{i-1})^* (A^{i})^* \equiv 0$, we have 
$\pi^i (A^{i})^* = (A^{i})^*$ and then
\begin{equation*} 
\Big( (I-\pi^i) \Big(f - ({\mathcal N}^i)'_{|u^{(0)} } v  \Big),(A^{i})^*w \Big)_{i} = 0
\end{equation*}
for all $w \in C^\infty_{E^{i+1}} ({\mathcal X})$, i.e.
\begin{equation*} 
 (I-\pi^i) \Big(f - ({\mathcal N}^i)'_{|u^{(0)} } v\Big)\in 
C^{2(s-1)+k,\lambda,s-1,\frac{\lambda}{2}} _{E^i} ({\mathcal X}_T)\cap
\mathcal{D}_{A^{i}} \cap
\mathcal{S}_{A^{i}}.
 \end{equation*}
Moreover, since ${\mathcal H}^i \subset 
C^{2(s-1)+k,\lambda,s-1,\frac{\lambda}{2}} _{E^i} ({\mathcal X}_T)\cap
\mathcal{S}_{(A^{i-1})^*} $, formula \eqref{eq.proj.1} implies that 
\begin{equation*} 
 \Pi^i(I-\pi^i) \Big(f - ({\mathcal N}^i)'_{|u^{(0)} } v\Big) =0.
 \end{equation*}
Then, according to Lemma \ref{l.Phi.t} and Corollary \ref{c.Hoelder.d.t}, the section $p$ 
belongs to the space $C^{2(s-1)+k,\lambda,s-1,\frac{\lambda}{2}} _{E^{i-1}} 
({\mathcal X}_T)\cap\mathcal{D}_{A^{i-1}} \ominus {\mathcal H}^{i-1}$ and satisfies 
\begin{equation} 
\label{eq.p.lin.sol}
   A^{i-1} p =  (I-\pi^i) \Big(f - ({\mathcal N}^i)'_{|u^{(0)} } v  \Big), \,\, 
	 (A^{i-2})^* p =0. 
\end{equation}
Thus, adding \eqref{eq.p.lin.sol} to \eqref{eq.sol.v}, we conclude that
 the pair $(u,p)$ is a solution to \eqref{eq.NS.complex.lin} 
from space \eqref{eq.sol.smoothness} that was to be proved.
\end{proof}

Now  we note that the scale of H\"older spaces 
$C^{2s+k, \lambda,s,\frac{\lambda}{2}}_{E} ({\mathcal X}_T)$ 
is coherent with the elliptic and parabolic linear theories but it does not take 
into account the structure of nonlinearity ${\mathcal N}^i$. We are going to 
slightly modify the scale 
$C^{2s+k,s,\lambda,\frac{\lambda}{2}}_{E} (\mathcal{X}_T)$ 
by introducing an additional H\"{o}lder exponent $\lambda'$ in order  to
gain some `smoothness' in $t$. Namely, for $s, k \in \mathbb{Z}_{\geq 0}$ and 
$0 < \lambda < \lambda' < 1$,
we introduce
\begin{equation} \label{eq.mathfrak.C}
   \mathfrak{C}^{k,\mathbf{s} (s,\lambda,\lambda')}_{E} (\mathcal{X}_T)
 :=   C^{2s+k+1,s,\lambda, \frac{\lambda}{2}}_{E} (\mathcal{X}_T) \cap
   C^{2s+k,s,\lambda', \frac{\lambda'}{2}}_{E} (\mathcal{X}_T)
\end{equation}
When given the norm
\begin{equation*}
   \| u \|_{\mathfrak{C}^{k,\mathbf{s} (s,\lambda,\lambda')}_{E} (\mathcal{X}_T)}
 :=
   \| u \|_{C^{2s+k+1,s,\lambda, \frac{\lambda}{2}}_{E} (\mathcal{X}_T)}
 + \| u \|_{C^{2s+k,s,\lambda', \frac{\lambda'}{2}}_{E} (\mathcal{X}_T)}.
\end{equation*}
this is obviously a Banach space.
To certain extent these spaces are similar to those with two-norm convergence which are
of key importance for ill-posed problems.

\begin{corollary}
\label{c.mathfrak.compact}
Let $s,k \in \mathbb{N}$,  $0 < \lambda < \lambda' < 1$. 
The following embedding is compact: 
\begin{equation*}
   \mathfrak{C}^{k,\mathbf{s} (s,\lambda,\lambda')}_{E} (\mathcal{X}_T)
 \hookrightarrow
   \mathfrak{C}^{k+1,\mathbf{s} (s-1,\lambda,\lambda')}_{E} (\mathcal{X}_T).
\end{equation*}
 \end{corollary}

\begin{proof} By Theorem \ref{t.emb.hoelder.t}, we have 
1) the space $C^{2s+k+1,\lambda,s,\frac{\lambda}{2}}_{E} (\mathcal{X}_T)$ 
is  embedded compactly into the space 
$C^{2(s-1)+k+1,\lambda',s-1,\frac{\lambda'}{2}}_{E} (\mathcal{X}_T)$ 
since $s+\lambda > s-1+\lambda'$;  2) the space $C^{2s+k,\lambda',s,\frac{\lambda'}{2}}_{E} (
\mathcal{X}_T)$ is  embedded compactly into $C^{2s+k,\lambda,s,\frac{\lambda}{2}}_{E} (
\mathcal{X}_T)$ for $0 < \lambda < \lambda'$; 
3) the space $C^{2s+k,\lambda,s,\frac{\lambda}{2}}_{E} (\mathcal{X}_T)$ 
is embedded continuously into the space 
$C^{2(s-1)+k+2,\lambda,s-1,\frac{\lambda}{2}}_{E} (\mathcal{X}_T)$. 

Hence it follows that if $S$ is a bounded set in the space
$\mathfrak{C}^{k,\mathbf{s} (s,\lambda,\lambda')}_{E} (\mathcal{X}_T)$ given by 
\eqref{eq.mathfrak.C} 
then any sequence from $S$ has a subsequence  converging in the space
\begin{equation*} 
   \mathfrak{C}^{k+1,\mathbf{s} (s-1,\lambda,\lambda')}_{E} (\mathcal{X}_T) :=
   C^{2(s-1)+k+2,s-1,\lambda, \frac{\lambda}{2}}_{E} (\mathcal{X}_T) \cap
   C^{2(s-1)+k+1,s-1,\lambda', \frac{\lambda'}{2}}_{E} (\mathcal{X}_T), 
\end{equation*}
as desired.
\end{proof}

\begin{remark} \label{r.mathfrak.C}
As the space $ \mathfrak{C}^{k,\mathbf{s} (s,\lambda,\lambda')}_{E^i} (\mathcal{X}_T)$
is defined as an intersection of the spaces from the scale 
$C^{2s'+k,s',\lambda, \frac{\lambda}{2}}_{E^i} (\mathcal{X}_T)$, 
we conclude that all the results on the continuity for elliptic 
and parabolic potentials are still valid for it, too. 
\end{remark} 

Thus, we denote by $ \mathfrak{C}^{k,\mathbf{s} (s,\lambda,\lambda')}_{E^i} (\mathcal{X}_T)
\cap \mathcal{D}_{A^{i}} $ the Banach space  
\begin{equation*} 
   \mathfrak{C}^{k,\mathbf{s} (s,\lambda,\lambda')}_{E^i} (\mathcal{X}_T)
 :=   \Big(C^{2s+k+1,s,\lambda, \frac{\lambda}{2}}_{E^i} (\mathcal{X}_T) 
\cap \mathcal{D}_{A^{i}}\Big) \cap \Big(
   C^{2s+k,s,\lambda', \frac{\lambda'}{2}}_{E^i} (\mathcal{X}_T)\cap \mathcal{D}_{A^{i}}\Big)
\end{equation*}
with the norm
\begin{equation*}
   \| u \|_{\mathfrak{C}^{k,\mathbf{s} (s,\lambda,\lambda')}_{E^i} (\mathcal{X}_T) 
\cap \mathcal{D}_{A^{i}}} :=   \| u \|_{C^{2s+k+1,s,\lambda, \frac{\lambda}{2}}_{E^i}
 (\mathcal{X}_T)\cap \mathcal{D}_{A^{i}}} + \| u \|_{C^{2s+k,s,\lambda', \frac{\lambda'}{2}}
_{E^i} (\mathcal{X}_T)\cap \mathcal{D}_{A^{i}}}.
\end{equation*}

\begin{lemma}
\label{l.map.V0}
Let $s,k \in \mathbb N$ and    $0 < \lambda < \lambda' < 1$. 
If $u^{(0)} \in \mathfrak{C}^{k,\mathbf{s} (s,\lambda,\lambda')}\cap
\mathcal{D}_{A^{i}}$ then 
Leray-Helmholtz projection $\pi^i$, the Fr\'echet derivative $({\mathcal N}^i)'_{|u^{(0)} }$ 
and the fundamental solution $\varPsi_\mu^i$ induce linear bounded operators 
\begin{equation} \label{eq.Psipi.prime.comp.v}
\varPsi_\mu^{(i,v)} \pi^i: \mathfrak{C}^{k,\mathbf{s} (s-1,\lambda,\lambda')}_{E^i} 
({\mathcal X}_T)\cap \mathcal{D}_{A^{i}} 
\to  \mathfrak{C}^{k,\mathbf{s} (s,\lambda,\lambda')} _{E^i} ({\mathcal X}_T)\cap
\mathcal{D}_{A^{i}} \cap \mathcal{S}_{(A^{i-1})^*}, 
\end{equation}
\begin{equation} \label{eq.Psipi.prime.comp.in}
\varPsi_\mu^{(i,in)} :  C^{2s+k+2,\lambda} _{E^i} ({\mathcal X})
\cap \mathcal{S}_{(A^{i-1})^*} 
\to  \mathfrak{C}^{k,\mathbf{s} (s,\lambda,\lambda')} _{E^i} ({\mathcal X}_T)\cap
\mathcal{D}_{A^{i}} \cap \mathcal{S}_{(A^{i-1})^*}, 
\end{equation}
and linear compact operators
\begin{equation} \label{eq.mathfrak.N.prime.comp.0}
({\mathcal N}^i)'_{|u^{(0)} }: 
\mathfrak{C}^{k,\mathbf{s} (s,\lambda,\lambda')} _{E^i} ({\mathcal X}_T)\cap
\mathcal{D}_{A^{i}} \to 
\mathfrak{C}^{k,\mathbf{s} (s-1,\lambda,\lambda')} _{E^i} ({\mathcal X}_T)\cap
\mathcal{D}_{A^{i}} .
\end{equation}
\begin{equation} \label{eq.mathfrak.N.prime.comp}
\varPsi_\mu^{(i,v)} \pi^i ({\mathcal N}^i)'_{|u^{(0)} }: 
\mathfrak{C}^{k,\mathbf{s} (s,\lambda,\lambda')} _{E^i} ({\mathcal X}_T)\cap
\mathcal{D}_{A^{i}} \to 
\mathfrak{C}^{k,\mathbf{s} (s,\lambda,\lambda')} _{E^i} ({\mathcal X}_T)\cap
\mathcal{D}_{A^{i}} \cap \mathcal{S}_{(A^{i-1})^*}.
\end{equation}
\end{lemma}

\begin{proof} As operators \eqref{eq.der.N.Dom}, \eqref{eq.Psipi.prime.bound.v} and 
\eqref{eq.PsiN.prime.bound} are bounded and linear, then operators 
 \eqref{eq.Psipi.prime.comp.v}, \eqref{eq.mathfrak.N.prime.comp.0}, 
\eqref{eq.mathfrak.N.prime.comp} are bounded, too (see Remark \ref{r.mathfrak.C}). 
The compactness of operators \eqref{eq.mathfrak.N.prime.comp.0} and 
\eqref{eq.mathfrak.N.prime.comp} follows  from the continuity 
of operators \eqref{eq.der.N.Dom}, 
\eqref{eq.PsiN.prime.bound} and the compact 
embedding described in Corollary \ref{c.mathfrak.compact}. 

Next, according to Lemma \ref{l.heat.key1},  the operator 
$\varPsi_\mu^{(i,in)}$ maps the space  $C^{2s+k+2,\lambda} _{E^i} ({\mathcal X})$ 
continuously to $C^{2s+k+2,\lambda,s,\frac{\lambda}{2}} _{E^i} ({\mathcal X}_T)$  
and  Theorem \ref{t.emb.hoelder.t} provides the continuous embedding of the last space 
to $C^{2s+k+1,\lambda,s,\frac{\lambda}{2}} _{E^i} ({\mathcal X}_T) 
\cap \mathcal{D}_{A^{i}}$. On the other hand, by Theorem \ref{t.emb.hoelder}, 
the space  $C^{2s+k+2,\lambda} _{E^i} ({\mathcal X})$ 
is continuously embedded to  $C^{2s+k+1,\lambda'} _{E^i} ({\mathcal X})$ and 
Lemma \ref{l.heat.key1} implies that  the operator 
$\varPsi_\mu^{(i,in)}$ maps the space  $C^{2s+k+1,\lambda'} _{E^i} ({\mathcal X})$ 
continuously to $C^{2s+k+1,\lambda',s,\frac{\lambda'}{2}} _{E^i} ({\mathcal X}_T)$.
Finally, it follows from Theorem \ref{t.emb.hoelder.t} that 
$C^{2s+k+1,\lambda',s,\frac{\lambda'}{2}} _{E^i} ({\mathcal X}_T)$ 
is continuously embedded to  $C^{2s+k,\lambda's, \frac{\lambda'}{2}} 
_{E^i} ({\mathcal X}) \cap \mathcal{D}_{A^{i}}$. Hence 
the operator \eqref{eq.Psipi.prime.comp.in} is bounded, too.
\end{proof}

\begin{theorem}
\label{t.invertible.psi}
Let $s,k \in \mathbb N$,  $0 < \lambda < \lambda' < 1$,
and $u^{(0)}\in \mathfrak{C}^{k,\mathbf{s} (s,\lambda,\lambda')} _{E^{i}}
(\mathcal{X}_T) \cap \mathcal{D}_{A^{i}}$. Then the following operator 
is continuously invertible: 
\begin{equation}
\label{eq.heat.pseudo.d2}
   I \! + \! \varPsi_\mu^{(i,v)} \pi^i V_0 ^i :
   \mathfrak{C}^{k,\mathbf{s} (s,\lambda,\lambda')} _{E^{i}} (\mathcal{X}_T)
 \cap \mathcal{D}_{A^{i}} \cap \mathcal{S}_{(A^{i-1})^*} 
 \to   \mathfrak{C}^{k,\mathbf{s}  (s,\lambda,\lambda')} _{E^{i}}(\mathcal{X}_T)
\cap \mathcal{D}_{A^{i}} 
\cap \mathcal{S}_{(A^{i-1})^*}.
\end{equation}
\end{theorem}

\begin{proof}
First we observe by Lemma \ref{l.map.V0} that the operator
$(I + \varPsi_\mu^{(i,v)} \pi^i V_0^i)$ is a continuous selfmapping of
$   \mathfrak{C}^{k,\mathbf{s} (s,\lambda,\lambda')} _{E^i}(\mathcal{X}_T)
\cap \mathcal{D}_{A^{i}} \cap \mathcal{S}_{(A^{i-1})^*} $. 
Our next goal is to show that this mapping is one-to-one.

\begin{lemma}
\label{l.unique.pseudo}
Let $s,k \in \mathbb N$, $u^{(0)} \in C^{2s+k,\lambda,s,\frac{\lambda}{2}}  
_{E^i}(\mathcal{X}_T) \cap\mathcal{D}_{A^{i}}$. 
If $v\in C^{2s+k,\lambda,s,\frac{\lambda}{2}} _{E^i}(\mathcal{X}_T)
\cap\mathcal{D}_{A^{i}} \cap \mathcal{S}_{(A^{i-1})^*}$  
satisfies $(I + \varPsi^i_{\mu}\pi^i V^i_0) (v) = 0$, then it is identically zero.
\end{lemma}

\begin{proof}
Indeed, using Lemma \ref{l.reduce.psi} we deduce that
the pair $(u,p)$, with the entry $p$ given by \eqref{eq.p.lin},   
belongs to the space ${\mathcal B}^{k,s,\lambda}_{i,1} $. 
%\begin{equation*}
%C^{2s+k,\lambda,s,\frac{\lambda}{2}} _{E^{i}} ({\mathcal X}_T)\cap
%\mathcal{D}_{A^{i}} \cap \mathcal{S}_{(A^{i-1})^*}  \times 
%C^{2(s-1)+k,\lambda,s-1,\frac{\lambda}{2}} _{E^{i-1}} ({\mathcal X}_T)\cap
%\mathcal{D}_{A^{i}} \cap \mathcal{S}_{(A^{i-2})^*}.
%\end{equation*}
It is a solution to the linearized Navier-Stokes Equation \eqref{eq.NS.complex.lin} 
with the zero data $f$ and $u_0$. 
Finally, according to the Uniqueness Theorem \ref{t.NS.deriv.unique}, 
we see that $v \equiv 0$. 
\end{proof}

Let us finish the proof of Theorem \ref{t.invertible.psi}.
According to Lemma \ref{l.map.V0}, the operator \eqref{eq.heat.pseudo.d2} is Fredholm and
its index equals to zero. Then the statement of the corollary follows from Lemma 
\ref{l.unique.pseudo} and Fredholm theorems.
\end{proof}

Let us introduce  the following Banach spaces: ${\mathfrak B}^{k,s,\lambda,\lambda'}_{i,1}=$ 
\begin{equation} \label{eq.smoothness.sol}
\mathfrak{C}^{k,\mathbf{s} (s,\lambda,\lambda')} _{E^i} 
	(\mathcal{X}_T)\cap \mathcal{D}_{A^{i}}\cap   \mathcal{S}_{(A^{i-1})^\ast} 
 \times
   \mathfrak{C}^{k,\mathbf{s} (s\!-\!1,\lambda,\lambda')} _{E^{i-1}} 
(\mathcal{X}_T) \cap \mathcal{D}_{A^{i-1}} \cap   \mathcal{S}_{(A^{i-2})^\ast}
\ominus C^{s,\frac{\lambda'}{2}}([0,T],\mathcal{H}^{i-1}), 
\end{equation}
\begin{equation} \label{eq.smoothness.data}
{\mathfrak B}^{k,s,\lambda,\lambda'}_{i,2} =\Big( \mathfrak{C}^{k,\mathbf{s} (s-1,\lambda,
\lambda')} _{E^i}
	(\mathcal{X}_T ) \cap \mathcal{D}_{A^{i}}  \Big) \times\Big(
   C^{2s+k+2,\lambda}  _{E^i} (\mathcal{X}) \cap \mathcal{S}_{(A^{i-1})^\ast}\Big).
\end{equation}
By the very definition, the space ${\mathfrak B}^{k,s,\lambda,\lambda'}_{i,j}$ 
is continuously embedded to ${\mathcal B}^{k,s,\lambda}_{i,j}$.

\begin{corollary}
\label{c.NS.deriv.unique}
Assume that $s,k \in \mathbb N$,  $0 < \lambda < \lambda' < 1$ 
   and $u^{(0)} \in \mathfrak{C}^{k,\mathbf{s} (s,\lambda,\lambda')} _{E^i} 
	(\mathcal{X}_T)\cap \mathcal{D}_{A^{i}} $. Then \eqref{eq.NS.complex.lin} 
	induces bounded linear continuously invertible operator ${\mathfrak A}^i_{\rm lin}$ 
between the Banach spaces ${\mathfrak B}^{k,s,\lambda,\lambda'}_{i,1}$ 
and ${\mathfrak B}^{k,s,\lambda,\lambda'}_{i,2}$.
\end{corollary}

\begin{proof} Indeed, As $A^{i} \circ A^{i-1}\equiv 0$, the operator 
\begin{equation} \label{eq.Ai.bound}
A^{i-1}: \mathfrak{C}^{k,\mathbf{s} (s\!-\!1,\lambda,\lambda')} _{E^{i-1}} 
(\mathcal{X}_T) \cap \mathcal{D}_{A^{i-1}} \cap   \mathcal{S}_{(A^{i-2})^\ast}
%\ominus C^{s,\frac{\lambda'}{2}}([0,T],\mathcal{H}^{i-1}) 
\to \mathfrak{C}^{k,\mathbf{s} (s\!-\!1,\lambda,\lambda')} _{E^i} 
(\mathcal{X}_T) \cap \mathcal{D}_{A^{i}}
\end{equation}
is linear and bounded. Moreover,  by \eqref{eq.Hodge.4}, we have 
$A^i L_\mu ^i = L_\mu ^{i+1} A^{i}$ an then 
\begin{equation} \label{eq.Li.bound}
L_\mu^i:
\mathfrak{C}^{k,\mathbf{s} (s,\lambda,\lambda')} _{E^i} 
	(\mathcal{X}_T)\cap \mathcal{D}_{A^{i}}\cap   \mathcal{S}_{(A^{i-1})^\ast}
\to \mathfrak{C}^{k,\mathbf{s} (s\!-\!1,\lambda,\lambda')} _{E^i} 
(\mathcal{X}_T) \cap \mathcal{D}_{A^{i}}
\end{equation}
is linear and bounded, too, because of Lemma \ref{l.diff.oper}. 
The continuity of the operator 
\begin{equation*} 
(L_\mu ^i  + ({\mathcal N}^i)'_{|u^{(0)}},  A^{i-1}) : 
{\mathfrak B}^{k,s,\lambda,\lambda'}_{i,1} 
\to \mathfrak{C}^{k,\mathbf{s} (s\!-\!1,\lambda,\lambda')} _{E^i} 
(\mathcal{X}_T) \cap \mathcal{D}_{A^{i}},
\end{equation*} 
 induced by \eqref{eq.NS.complex.lin}, follows from Lemma 
\ref{l.map.V0}. 

Next, by the definition the operator $\gamma_0$ maps the space 
$\mathfrak{C}^{k,\mathbf{s} (s,\lambda,\lambda')} _{E^i} 
(\mathcal{X}_T)\cap \mathcal{D}_{A^{i}}\cap   \mathcal{S}_{(A^{i-1})^\ast}$ 
continuously to the space $C^{2s+k+1,\lambda} _{E^i} ({\mathcal X})\cap \mathcal{D}_{A^{i}}
\cap   \mathcal{S}_{(A^{i-1})^\ast}$. But complex \eqref{eq.ellcomp} is elliptic and 
hence  Theorem \ref{t.Hoelder.Laplace} and \eqref{eq.Hodge.4}, \eqref{eq.Hodge.5} imply  
\begin{equation*} 
\gamma_0 u = A_i^* \varphi^{i+1} A_i \gamma_0 u
\in C^{2s+k+2,\lambda} _{E^i} ({\mathcal X})
\cap   \mathcal{S}_{(A^{i-1})^\ast}
\end{equation*}
and then the operator
\begin{equation} \label{eq.gamma.bound}
\gamma_0 : \mathfrak{C}^{k,\mathbf{s} (s,\lambda,\lambda')} _{E^i} 
(\mathcal{X}_T)\cap \mathcal{D}_{A^{i}}\cap   \mathcal{S}_{(A^{i-1})^\ast} \to 
C^{2s+k+2,\lambda} _{E^i} ({\mathcal X})
\cap   \mathcal{S}_{(A^{i-1})^\ast}
\end{equation}
is bounded.  
Therefore  the operator 
${\mathfrak A}^i_{\rm lin}$, induced by \eqref{eq.NS.complex.lin}, 
is bounded, too.

It is left to prove that for each pair $(f,u_0)$ from the space 
${\mathfrak B}^{k,s,\lambda,\lambda'}_{i,2}$, defined by \eqref{eq.smoothness.data}, 
there is a unique solution  $(u,p)$ to problem \eqref{eq.NS.complex.lin} 
from the space ${\mathfrak B}^{k,s,\lambda,\lambda'}_{i,1}$, 
defined by \eqref{eq.smoothness.sol},  
and, moreover, 
\begin{equation*}
\|(u,p)\|_{{\mathfrak B}^{k,s,\lambda,\lambda'}_{i,1}} \leq 
c _\mu \|(f,u_0)\|_{{\mathfrak B}^{k,s,\lambda,\lambda'}_{i,2}}
\end{equation*}
with a positive constant $c _\mu$ independent on $(f,u_0)$. 
But the first statement  follows from Lemmata \ref{l.reduce.psi} and 
\ref{l.map.V0}, Theorems \ref{t.invertible.psi} and \ref{t.NS.deriv.unique} and 
the estimate follows from the Banach Closed Graph Theorem. 
\end{proof}

\section{The Navier-Stokes type equations as an open map}
\label{s.NS.OpenMap}

We plan to treat the Navier-Stokes equations  as a nonlinear injective
Fredholm operator with open range in proper Banach spaces.
Recall that a nonlinear operator $\mathfrak{A} : B_2 \to 
B_2$ in Banach spaces $B_1$, $B_2$ is called
Fredholm if it has a Frech\'et derivative at each point $x_0 \in B_1$ and this 
derivative is a Fredholm linear map from $B_1$ to $B_2$ (see \cite{Sm65}).
We begin with the Uniqueness Theorem for \eqref{eq.NS.complex}. 

\begin{theorem}
\label{t.NS.unique}
Suppose that \eqref{eq.property.M1} holds. Then for each pair
  $(f,u_0)  \in  C^{0,0,0,0} _{E^i} (\mathcal{X}_T)
 \times   C^{0,0}_{E^i} (\mathcal{X}) \cap \mathcal{S}_{(A^{i-1})^\ast}$ 
nonlinear Navier-Stokes type equations \eqref{eq.NS.complex} 
have at most one solution in the space
  $ C^{2,0,1,0} _{E^i}({\mathcal{X}_T)} \cap \mathcal{S}_{(A^{i-1})^\ast} \times
   C^{0,0,0,0} _{E^{i-1}}(\mathcal{X}_T) \cap 
\mathcal{D}_{A^{i-1}} $, satisfying \eqref{eq.p.ort}. 
\end{theorem}

\begin{proof} Again, one may follow
   the original paper \cite{Lera34a} or
   the proof of %Theorem 3.2 for $n = 2$ and 
Theorem 3.4 for $n = 3$ in \cite{Tema79},
showing the uniqueness result by integration by parts. 
%Cf. also Theorem \ref{t.NS.deriv.unique}.

Indeed, let $(u',p'')$ and $(u'',p'')$
be any two solutions to \eqref{eq.NS.complex} from the declared function space. 
Then for the difference
   $(u,p) = (u' - u'', p' - p'')$
we get
\begin{equation}
\label{eq.NS.D0}
 L^i_{\mu} u + A^{i-1} p  = {\mathcal N}^i (u'') - {\mathcal N}^i (u') .
\end{equation}
As $ u \in C^{2,0,1,0}_{E^i} (\mathcal{X}_T ) \cap \mathcal{S}_{(A^{i-1})^\ast}$, $p
 \in   C^{0,0,0,0} (\mathcal{X}_T) \cap \mathcal{D}_{A^{i-1}}$,
 the sections 
   $u$,
   $A^i u$,
   $\partial_t u$,
   $L^i_{\mu} u$ and
   $ A^{i-1} p$
are square integrable over all of $\mathcal{X}$ for each fixed 
$t \in [0,T]$. Furthermore, the integrals
   $(\mathcal{N}^i (u'), u)_{i}$ and
   $(\mathcal{N}^i (u''), u)_{i}$
converge, for both  $\mathcal{N}^i (u')$ and $\mathcal{N}^i (u'')$
are  of class
   $C^{0,0,0,0}_{E^i} (\mathcal{X}_T)$ (see Lemmata \ref{l.product} and \ref{l.diff.oper}). 
Then 
\begin{equation} \label{eq.dt}
   \partial_t \| u (\cdot, t) \|^2_{i}
 = 2\, (\partial_t u, u)_{i},
\end{equation}
and, since $(A^{i-1})^* u=0$ we easily obtain,  
\begin{equation} \label{eq.Lmu}
   (L^i_{\mu} u + A^{i-1} p, u)_{i}  = 
      \frac{1}{2}\, \partial_t \| u (\cdot, t) \|^2_{i}
 + \,\, \mu \| A^{i} u (\cdot, t) \|^2_{i+1} .
\end{equation}
Therefore   \eqref{eq.NS.D0}, \eqref{eq.dt} and \eqref{eq.Lmu}   imply
\begin{equation*}
  \frac{1}{2}  \partial_t\, \| u (\cdot,t) \|^2_{i}
 + \mu  \| A^i u (\cdot,t) \|^2_{i+1} 
=     (\mathcal{N}^i (u''), u)_{i} - (\mathcal{N}^i (u'), u)_{i}.
 \end{equation*}
As $(A^{i-1})^* u=0$, trivial verification shows that
\begin{equation*}
(A^{i-1} {\mathcal M}_{i,2}(u',u') ,u)_{i} = 
(A^{i-1} {\mathcal M}_{i,2}(u'',u'') ,u)_{i} =0.
\end{equation*}
Then, as the form ${\mathcal M}_{i,1}$ is bilinear, we obtain
\begin{equation*}
   (\mathcal{N}^i (u'), u)_{i} - (\mathcal{N}^i (u''), u)_{i} =
% \end{equation*}
%\begin{equation*}
  (\mathcal{M}_{i,1} (A^i u',u'), u)_{i}
-  (\mathcal{M}_{i,1} (A^i u'',u''), u)_{i}=
 \end{equation*}
\begin{equation*}
  (\mathcal{M}_{i,1} (A^i u,u'), u)_{i}+  (\mathcal{M}_{i,1} (A^i u'',u), u)_{i},
 \end{equation*}
and so, for all $t \in [0,T]$, 
\begin{equation*}
   \partial_t\, \| u (\cdot,t) \|^2_{i}
 +2 \mu  \| A^i u (\cdot,t) \|^2_{i+1} =  
%\end{equation*} 
%\begin{equation*}
2(\mathcal{M}_{i,1} (A^i u,u'), u)_{i}
+  2(\mathcal{M}_{i,1} (A^i u'',u), u)_{i}.
\end{equation*}

As $u',u''$ and $u$ are of class $C^{2,0, 1,0} _{E^i} (\mathcal{X}_T)$,
by property \eqref{eq.property.M1} 
%of the form $\mathcal{M}_{i,1}$ 
we have
\begin{equation*}
 2| (\mathcal{M}_{i,1} (A^i u,u'), u)_{i}|
  \leq 
     2 c ({\mathcal M})\| A ^i u \|_{i+1}
   \| u\|_{i} \| u'\|_{C^{0,0,0,0} _{E^{i}} (\mathcal{X}_T)}\leq 
\end{equation*}
\begin{equation*}
2\mu       \| A ^i u \|^2_{i+1}
  +\frac{c ({\mathcal M})^2}{2\mu} \| u\|^2_{i} 
	\| u'\|^2_{C^{0,0,0,0} _{E^{i}} (\mathcal{X}_T)},
\end{equation*}
\begin{equation*}
 | (\mathcal{M}_{i,1} (A^i u'',u), u)_{i}|
  \leq 
     c ({\mathcal M}) \| A ^i u'' \|_{C^{0,0,0,0} _{E^{i+1}}  (\mathcal{X}_T)}
   \| u\|^2_{i} ,
\end{equation*}
whence for all $t \in [ 0,T]$ we have 
\begin{equation*}
  \partial_t\, \| u (\cdot,t) \|^2_{i}\leq
\Big( \frac{c ({\mathcal M})^2}{2\mu}\| u'\|^2_{C^{0,0,0,0} _{E^{i}} (\mathcal{X}_T)} + 
c ({\mathcal M} )\| A ^i u'' \|_{C^{0,0,0,0} _{E^{i+1}}  (\mathcal{X}_T)} \Big)\,   
\| u (\cdot,t) \|^2_{i}.
\end{equation*}
Now we note that from the inequality
   $x' (t) \leq z (t) x (t)$
for all $t$ in some interval of the real axis it follows that
\begin{equation*}
   \frac{d}{dt} \left( e^{- Z (t)} x (t) \right) \leq 0,
\end{equation*}
where $Z$ is a primitive function for $z$.
Therefore, since $Z (t) = 2c\, t$  is a primitive for the function  $z (t) = 2c$, we conclude
that, for all $t \in (0,T]$, 
\begin{equation*}
\frac{d}{dt} \Big( e^{- 2c\, t} \| u (\cdot, t) \|^2_{i} \Big)
 \leq 0.
\end{equation*}

Pick any $t \in (0,T]$. Then 
\begin{equation*}
   \int_0^t
   \frac{d}{ds}
   \Big( e^{-2c\, s} \| u (\cdot, s) \|^2_{i} \Big) ds
  = 
   e^{-2c\, t} \| u (\cdot,t) \|^2_{i}
 - \| u (\cdot,0) \|^2_{i}
 = 
\end{equation*}
 \begin{equation*}
  e^{- 2c\, t} \| u (\cdot,t) \|^2_{i}
 \leq    0
\end{equation*}
because $u (x,0) = 0$ for all $x \in \mathcal{X}$. 
Thus,  $u \equiv 0$ because 
\begin{equation*}
   \| u (\cdot,t) \|^2_{i} \leq   0 \mbox{ for all } t \in [0,T]. 
\end{equation*}
It follows that $A^{i-1}p (\cdot,t) \equiv 0$ for each $t \in [0,T]$. 
As $p$ is $L^2_{E^i} (\mathcal{X})$-orthogonal to $\mathcal{H}^{i-1}$ we see that 
$p \equiv 0$, i.e., the solutions $(u',p')$ and $(u'',p'')$ coincide,  as desired.
\end{proof}

The nonlinear Navier-Stokes equations can be reduced to a non-linear pseudo-differential 
Fredholm type equation in much the same way as the corresponding
linearised equations. We proceed with an explicit description.

\begin{lemma}
\label{l.map.G}
Let $s,k\in \mathbb N$ and
   $0 < \lambda < \lambda' < 1$. The 
Leray-Helmholtz projection $\pi^i$ and 
the fundamental solution $\varPsi_\mu^i$ induce nonlinear continuous operators
\begin{equation} \label{eq.PsiN.bound}
\varPsi_\mu^{(i,v)} \pi^i {\mathcal N}^i : 
C^{2s+k,\lambda,s,\frac{\lambda}{2}} _{E^i} ({\mathcal X}_T)\cap
\mathcal{D}_{A^{i}} \to 
C^{2(s+1)+k-1,\lambda,s+1,\frac{\lambda}{2}} _{E^i} ({\mathcal X}_T)\cap
\mathcal{D}_{A^{i}} \cap \mathcal{S}_{(A^{i-1})^*}
\end{equation}
and nonlinear continuous compact operators
\begin{equation} \label{eq.N.comp}
{\mathcal N}^i : 
\mathfrak{C}^{k,\mathbf{s} (s,\lambda,\lambda')} _{E^i} ({\mathcal X}_T)\cap
\mathcal{D}_{A^{i}} \to 
\mathfrak{C}^{k,\mathbf{s} (s-1,\lambda,\lambda')} _{E^i} ({\mathcal X}_T)\cap
\mathcal{D}_{A^{i}} . 
\end{equation}
\begin{equation} \label{eq.PsiN.comp}
\varPsi_\mu^{(i,v)} \pi^i {\mathcal N}^i : 
\mathfrak{C}^{k,\mathbf{s} (s,\lambda,\lambda')} _{E^i} ({\mathcal X}_T)\cap
\mathcal{D}_{A^{i}} \to 
\mathfrak{C}^{k,\mathbf{s} (s,\lambda,\lambda')} _{E^i} ({\mathcal X}_T)\cap
\mathcal{D}_{A^{i}} \cap \mathcal{S}_{(A^{i-1})^*}. 
\end{equation}
\end{lemma}

\begin{proof}
The continuity of operators \eqref{eq.PsiN.bound}, \eqref{eq.PsiN.comp}
follows from Lemmata \ref{l.N.cont} and \ref{l.heat.key1}.
Then the compactness of  operators \eqref{eq.N.comp} and \eqref{eq.PsiN.comp}  follows  from 
the continuity  of operators \eqref{eq.PsiN.bound}, \eqref{eq.N.bound} and the compact 
embedding described in Corollary \ref{c.mathfrak.compact}. 
\end{proof}

The following lemma is just an adaptation to the particular function spaces of 
the standard reduction of the Navier-Stokes types 
equation to a pseudo-differential equation that does not involve 
the `pressure' $p$. 

\begin{lemma}
\label{l.reduce.psi.n}
Suppose that $s, k \in \mathbb N$ and $0<\lambda<\lambda'<1$.  
Let moreover $(f,u_0)$ be an arbitrary pair of ${\mathcal B}^{k,s,\lambda}
_{i,2}$, defined by 
\eqref{eq.data.smoothness}, 
Then there is a solution $ (u,p) $ of class ${\mathcal B}^{k,s,\lambda}
_{i,1}$, defined by \eqref{eq.sol.smoothness}, 
 to problem \eqref{eq.NS.complex} if and only if there is a solution 
$v \in  C^{2s+k,\lambda,s,\frac{\lambda}{2}} _{E^i} ({\mathcal X}_T) \cap
\mathcal{D}_{A^{i}} \cap \mathcal{S}_{(A^{i-1})^*}$ to 
pseudo-differential equation 
\begin{equation}
\label{eq.NS.lin.psi.n}
  v + \varPsi_\mu^{(i,v)} \pi^i {\mathcal N}^i v = F
\end{equation}
with the datum $F=\varPsi_\mu^i (\pi^i f, u_0) \in 
C^{2s+k,\lambda,s,\frac{\lambda}{2}} _{E^i} ({\mathcal X}_T)\cap
\mathcal{D}_{A^{i}} \cap \mathcal{S}_{(A^{i-1})^*}$. Besides,  
\begin{equation} 
\label{eq.p.lin.n}
 u=v, \,\,  p = \varPhi_{i-1}  (I-\pi^i) \Big(f - {\mathcal N}^i v  \Big).
\end{equation}
\end{lemma}

\begin{proof} It is similar to the proof of Lemma \ref{l.reduce.psi.n}. 
\end{proof}

We are already in a position to state an open mapping theorem for %the reduced equation 
\eqref{eq.NS.lin.psi.n}. 

\begin{theorem}
\label{t.OpenMap.psi}
Let $s,k \in \mathbb N$,  $0 < \lambda < \lambda' < 1$. Then the following mapping 
is continuous, Fredholm,  injective and open:
\begin{equation}\label{eq.heat.pseudo.nn}
   I \! + \! \varPsi_\mu^{(i,v)} \pi^i {\mathcal N} ^i :
   \mathfrak{C}^{k,\mathbf{s} (s,\lambda,\lambda')} _{E^{i}} (\mathcal{X}_T)
	 \cap \mathcal{D}_{A^{i}} \cap \mathcal{S}_{(A^{i-1})^*} 
 \to   \mathfrak{C}^{k,\mathbf{s}  (s,\lambda,\lambda')} _{E^{i}}(\mathcal{X}_T)
\cap \mathcal{D}_{A^{i}} 
\cap \mathcal{S}_{(A^{i-1})^*}.
\end{equation}
\end{theorem}

\begin{proof}
First we note that Lemma \ref{l.map.G} implies that the
operator $\varPsi_\mu^{(i,v)} \pi^i {\mathcal N} ^i $ maps the space 
$
  \mathfrak{C}^{k,\mathbf{s} (s,\lambda,\lambda')} _{E^{i}} (\mathcal{X}_T)
	 \cap \mathcal{D}_{A^{i}} \cap \mathcal{S}_{(A^{i-1})^*} 
$
continuously and compactly into itself. In particular, the mapping 
\eqref{eq.heat.pseudo.nn} is continuous. 
We now turn to the one-to-one property of the mapping \eqref{eq.heat.pseudo.nn}.

\begin{lemma}
\label{l.unique.psi.n}
Let $s, k \in \mathbb N$  and  $0 < \lambda < 1$.
If $F \in C^{2s+k,\lambda,s,\frac{\lambda}{2}} _{E^i} ({\mathcal X}_T)\cap
\mathcal{D}_{A^{i}} \cap \mathcal{S}_{(A^{i-1})^*}$
then \eqref{eq.NS.lin.psi.n} has no more than one solution in
$C^{2s+k,\lambda,s,\frac{\lambda}{2}} _{E^i} ({\mathcal X}_T)\cap
\mathcal{D}_{A^{i}} \cap \mathcal{S}_{(A^{i-1})^*}$.
\end{lemma}

\begin{proof}
Suppose that $v', v'' \in  C^{2s+k,\lambda,s,\frac{\lambda}{2}} _{E^i} ({\mathcal X}_T)\cap
\mathcal{D}_{A^{i}} \cap \mathcal{S}_{(A^{i-1})^*}$
are two solutions to \eqref{eq.NS.lin.psi.n}. Using 
Lemma \ref{l.bound.heat.initial.hoelder} we conclude that both $v'$ and $v''$ 
are solutions to 
\begin{equation*} %\label{eq.sol.v.n.1}
   \left\{
   \begin{array}{rclll}
     L^i_{\mu} w +  \pi^i {\mathcal N}^i w  
   & =
   & L^i_{\mu}  F
   & \mbox{in}
   & \mathcal{X} \times (0,T) ,\\
	(A^{i-1})^*w & = & 0 & \mbox{in}
   & \mathcal{X} \times [0,T] ,
\\
     \gamma_0\, w (x)
   & =
   &  F (x,0) 
   & \mbox{on}
   & \mathcal{X}.
\end{array}
\right.
\end{equation*}
%\begin{equation} \label{eq.sol.v.n.2}
%   \left\{
%   \begin{array}{rclll}
%     L^i_{\mu} v'' +  \pi^i {\mathcal N}^i v''  
%   & =
%   & L^i_{\mu}  F
%   & \mbox{in}
%   & \mathcal{X} \times (0,T) ,\\
%	(A^{i-1})^*v'' & = & 0 & \mbox{in}
%   & \mathcal{X} \times [0,T] ,
%\\
%     \gamma_0\, v'' (x)
%   & =
%   &  F (x,0) 
%   & \mbox{on}
%   & \mathcal{X}.
%\end{array}
%\right.
%\end{equation}

Now Lemma \ref{l.reduce.psi.n} implies that the pair $(v',p')$ and  $(v'',p'')$
are solutions to \eqref{eq.NS.complex} with $f= L_\mu^i F$, $u_0 (x) = F(x,0)$, 
$x \in \mathcal X$ where $u'=v'$, $u''=v''$ and 
\begin{equation*} %\label{eq.p.lin.n.1}
   p' = \varPhi_{i-1}  (I-\pi^i) \Big(f - {\mathcal N}^i v'  \Big), \,\, 
	 p'' = \varPhi_{i-1}  (I-\pi^i) \Big(f - {\mathcal N}^i v''  \Big).
\end{equation*}
Thus, by the uniqueness of Theorem \ref{t.NS.unique}, we get $v' = v''$, $p' = p''$.
\end{proof}

Lemma \ref{l.unique.psi.n} implies immediately that the mapping in \eqref{eq.heat.pseudo.nn} 
is actually one-to-one. Now Theorem \ref{t.invertible.psi} shows that the Frech\'et derivative
$ (I + \varPsi_\mu ^i \pi^i {\mathcal N}^i)'_{|u^{(0)}}
$ of the mapping $(I + \varPsi_\mu ^i \pi^i {\mathcal N}^i)$ at an arbitrary point
$u ^{(0)} \in
   \mathfrak{C}^{k,\mathbf{s} (s,\lambda,\lambda')} _{E^i} ({\mathcal X}_T)\cap 
\mathcal{D}_{A^{i}} \cap \mathcal{S}_{(A^{i-1})^*}$ 
is a continuously invertible selfmapping of the space
$ \mathfrak{C}^{k,\mathbf{s} (s,\lambda,\lambda')} _{E^i} ({\mathcal X}_T) \cap 
\mathcal{D}_{A^{i}} \cap \mathcal{S}_{(A^{i-1})^*}$. In particular, 
the mapping \eqref{eq.heat.pseudo.nn} is Fredholm one. 
Then the  openness of \eqref{eq.heat.pseudo.nn} follows from the Implicit Mapping 
Theorem in Banach spaces, see for instance Theorem 5.2.3 of \cite[p.~101]{Ham82}.
\end{proof}

When combined with Lemma \ref{l.reduce.psi.n}, Theorem \ref{t.OpenMap.psi} implies that
the Navier-Stokes equations induce an open mapping in the function spaces under
consideration.

\begin{corollary}
\label{c.open.NS.short}
Let $s,k \in \mathbb N$,  $0 < \lambda < \lambda' < 1$. 
Then \eqref{eq.NS.complex} 
	induces continuous nonlinear open injective mapping ${\mathfrak A}^i$ 
from  ${\mathfrak B}^{k,s,\lambda,\lambda'}_{i,1}$ 
to ${\mathfrak B}^{k,s,\lambda,\lambda'}_{i,2}$, see \eqref{eq.smoothness.sol}, 
\eqref{eq.smoothness.data}.
\end{corollary}

\begin{proof} As operators \eqref{eq.Ai.bound}, \eqref{eq.Li.bound}, 
\eqref{eq.gamma.bound} \eqref{eq.N.comp} are continuous we conclude that  
the mapping ${\mathfrak A}^i$, 
induced by \eqref{eq.NS.complex}, maps  ${\mathfrak B}^{k,s,\lambda,\lambda'}_{i,1}$ 
 continuously to the space ${\mathfrak B}^{k,s,\lambda,\lambda'}_{i,2}$. 

It is left to prove that for any pair $(u^{(0)},p^{(0)})$ from 
${\mathfrak B}^{k,s,\lambda,\lambda'}_{i,1}$, there is $\delta > 0$ with the property 
that for all data $(f,u_0)$ from ${\mathfrak B}^{k,s,\lambda,\lambda'}_{i,2}$, 
satisfying the estimate
\begin{equation}
\label{eq.NS.open.est}
\| (f,u_0) - {\mathfrak A}^i(u^{(0)},p^{(0)}) 
   \|_{{\mathfrak B}^{k,s,\lambda,\lambda'}_{i,2}} < \delta, 
\end{equation}
nonlinear equations \eqref{eq.NS.complex} 
has a unique solution $(u,p)$ in ${\mathfrak B}^{k,s,\lambda,\lambda'}_{i,1}$. 

By Lemma Lemma \ref{l.map.V0}, 
the parabolic potential $\varPsi_{\mu}^i$ induces the bounded linear operator
\begin{equation}
\label{eq.oper.psi}
   \mathfrak{C}^{k,\mathbf{s} (s-1,\lambda,\lambda')}_{E^i} (\mathcal{X}_T) 
	\cap \mathcal{D}_{A^{i}} \times 
	C^{2s+k+2,\lambda} _{E^i} (\mathcal{X})
 \to
   \mathfrak{C}^{k,\mathbf{s} (s,\lambda,\lambda')} _{E^i}
	(\mathcal{X}_T) \cap \mathcal{D}_{A^{i}} .
\end{equation}

We now apply Lemma \ref{l.reduce.psi.n} to see that  $ u^{(0)}$ 
is a solution to the operator equation
\begin{equation*}
   (I + \varPsi_\mu^i \pi^i \mathcal{N}^i) u^{(0)} = g_0^{(0)}, 
\end{equation*}
with the right-hand side 
$g_0^{(0)} := \varPsi_{\mu}^i  ((L_{\mu}^i +
	+ \pi^i \mathcal{N}^i )u^{(0)}, \gamma_0 u^{(0)})$,
belonging to the space
   $ \mathfrak{C}^{k,\mathbf{s} (s,\lambda,\lambda')} _{E^i}
	(\mathcal{X}_T) \cap \mathcal{D}_{A^{i}} \cap \mathcal{S}_{(A^{i-1})^*} $.
Set $ g_0 = \varPsi_{\mu}^i (\pi ^i f,  u_0 )$
which belongs to
$  \mathfrak{C}^{k,\mathbf{s} (s,\lambda,\lambda')} _{E^i}
	(\mathcal{X}_T) \cap \mathcal{D}_{A^{i}} \cap \mathcal{S}_{(A^{i-1})^*}$ 
by Lemmata  \ref{l.heat.key1} and \ref{l.map.V0.bound}. 

An immediate calculations shows that
\begin{equation*}
   \| g_0 - g_0^{(0)}  \|_{\mathfrak{C}^{k,\mathbf{s} (s,\lambda,\lambda')}_{E^i} 
(\mathcal{X}_T)\cap \mathcal{D}_{A^{i}}  }
\leq \| \varPsi_{\mu}^i \| | (f,u_0) - {\mathfrak A}^i(u^{(0)},p^{(0)}) 
   \|_{{\mathfrak B}^{k,s,\lambda,\lambda'}_{i,2}}
\end{equation*}
where $\| \varPsi_{\mu} ^i\|$ is the norm of bounded linear operator \eqref{eq.oper.psi}.

If $\delta > 0$ in the estimate \eqref{eq.NS.open.est} is small enough, then
Theorem \ref{t.OpenMap.psi} shows that there is a unique solution
$   v \in \mathfrak{C}^{k,\mathbf{s} (s,\lambda,\lambda')}_{E^i} 
(\mathcal{X}_T)\cap \mathcal{D}_{A^{i}} \cap \mathcal{S}_{(A^{i-1})^*}$
to the operator equation 
$   v + \varPsi_{\mu}^i \pi^i \mathcal{N}^i v = g_0$. 
The pair
 $ u= v$, $   p  =
   \varPhi_{i-1} (I-\pi^i) \left( f - 
	{\mathcal N}^i u \right)$ 
belongs to ${\mathfrak B}^{k,s,\lambda,\lambda'}_{i,1}$ 
 satisfies nonlinear equations \eqref{eq.NS.complex},
   which is due to Lemma \ref{l.reduce.psi.n}.
Finally, the uniqueness of the solution $(u,p)$, 
follows from Theorem \ref{t.NS.unique}.
\end{proof}

We would like to spread the results to the case of infinity 
differentiable sections. With this purpose, let us prove a theorem on the improvement 
of the regularity for solutions to operator equation 
\eqref{eq.NS.lin.psi.n}.

\begin{theorem} \label{t.imp.reg}
Let  $s,k\in \mathbb N$, $0<\lambda\leq \lambda'<1$.
If $v \in {C}^{3,\lambda,1,\frac{\lambda}{2}}_{E^i} 
(\mathcal{X}_T ) \cap \mathcal{D}_{A^{i}} \cap \mathcal{S}_{(A^{i-1})^*}$ is 
a solution to \eqref{eq.NS.lin.psi.n} 
with datum $F\in {C}^{2s+k,\lambda,s,\frac{\lambda}{2}}_{E^i} 
(\mathcal{X}_T )  \cap \mathcal{D}_{A^{i}} \cap \mathcal{S}_{(A^{i-1})^*} $ then 
$v\in {C}^{2s+k,\lambda,s,\frac{\lambda}{2}}_{E^i} 
(\mathcal{X}_T )  \cap \mathcal{D}_{A^{i}} \cap \mathcal{S}_{(A^{i-1})^*} $, too. 
\end{theorem}

\begin{proof} For $s=k=1$ there is nothing to prove. 

Let first $s=1$ and $k>1$. According to  Lemma \ref{l.map.G}, the potential 
$\varPsi_{\mu}^i \pi ^i {\mathcal N}^i v $ belongs to the space 
${C}^{4,\lambda,2,\frac{\lambda}{2}}_{E^i} 
(\mathcal{X}_T ) \cap \mathcal{D}_{A^{i}} \cap \mathcal{S}_{(A^{i-1})^*}$. 
Again it follows from  \eqref{eq.NS.lin.psi.n}  that 
the section $v$ belongs to ${C}^{4,\lambda,1,\frac{\lambda}{2}}_{E^i} 
(\mathcal{X}_T ) \cap \mathcal{D}_{A^{i}} \cap \mathcal{S}_{(A^{i-1})^*}$. 

If $k=2$ then the statement of the theorem is proved for $s=1$. 
If $k>2$ then we may repeat the arguments. Indeed, 
 according to  Lemma \ref{l.map.G}, the potential 
$\varPsi_{\mu}^i \pi ^i {\mathcal N}^i v $ belongs to the space 
${C}^{5,\lambda,2,\frac{\lambda}{2}}_{E^i} 
(\mathcal{X}_T ) \cap \mathcal{D}_{A^{i}} \cap \mathcal{S}_{(A^{i-1})^*}$. 
It follows from  \eqref{eq.NS.lin.psi.n}  that 
the section $v$ belongs to ${C}^{5,\lambda,1,\frac{\lambda}{2}}_{E^i} 
(\mathcal{X}_T ) \cap \mathcal{D}_{A^{i}} \cap \mathcal{S}_{(A^{i-1})^*}$. 
Repeating the arguments $(k-1)$ times , we see that 
the section $v$ belongs actually 
to ${C}^{2+k,\lambda,1,\frac{\lambda}{2}}_{E^i} 
(\mathcal{X}_T ) \cap \mathcal{D}_{A^{i}} \cap \mathcal{S}_{(A^{i-1})^*}$. 
This proves the theorem for $s=1$ and $k>1$. 
  
Now we may finish the prove by the induction with respect to $s$. Assume 
that the statement of the theorem holds true for $s=s'$. Let us show that 
it is valid for $s=s'+1$, too. Indeed, let the datum $F$ belong to 
$C^{2(s'+1)+k,\lambda,s'+1,\frac{\lambda}{2}}_{E^i} 
(\mathcal{X}_T )  \cap \mathcal{D}_{A^{i}} \cap \mathcal{S}_{(A^{i-1})^*}$. 
By the Embedding Theorem \ref{t.emb.hoelder.t}
the solution $v$ belongs to the space 
$C^{2s'+k+2,\lambda,s',\frac{\lambda}{2}}_{E^i} 
(\mathcal{X}_T )  \cap \mathcal{D}_{A^{i}} \cap \mathcal{S}_{(A^{i-1})^*}$. 
Then, according to the inductive  assumption, the solution 
$v$ to \eqref{eq.NS.lin.psi.n} belongs to  
$C^{2s'+k+2,\lambda,s',\frac{\lambda}{2}}_{E^i} 
(\mathcal{X}_T )  \cap \mathcal{D}_{A^{i}} \cap \mathcal{S}_{(A^{i-1})^*}$, too. 
Now it follows from  Lemma \ref{l.map.G} that  the potential 
$\varPsi_{\mu}^i \pi ^i {\mathcal N}^i v $ belongs to the space 
${C}^{2(s'+1)+k+1,\lambda,s'+1,\frac{\lambda}{2}}_{E^i} 
(\mathcal{X}_T ) \cap \mathcal{D}_{A^{i}} \cap \mathcal{S}_{(A^{i-1})^*}$ and  
then \eqref{eq.NS.lin.psi.n}  implies that 
$v$ belongs to ${C}^{2(s'+1)+k,\lambda,s'+1,\frac{\lambda}{2}}_{E^i} 
(\mathcal{X}_T ) \cap \mathcal{D}_{A^{i}} \cap \mathcal{S}_{(A^{i-1})^*}$, 
which was to be proved. 
\end{proof}

Next,  fix a real number $\lambda' \in (0,1)$   
and consider the Fr\'echet spaces $ {C}^{\infty} _{E^i}(\mathcal{X}_T)$ and 
$ {C}^{\infty} _{E^i}(\mathcal{X})$   
endowed with the family of the semi-norms 
\begin{equation*}
\Big\{
\|\cdot\|_{\mathfrak{C}^{ \mathbf{s} (s,\lambda'/2,\lambda')} _{E^i}(\mathcal{X}_T)} 
\Big\}_{s\in \mathbb{Z}_+} \,\,
 \Big\{ \|\cdot\|_{C^{s,\lambda'}_{E^i} 
(\mathcal{X})} \Big\}_{s\in \mathbb{Z}_+},
\end{equation*}
respectively. Embedding Theorem \ref{t.emb.hoelder} and \ref{t.emb.hoelder.t} imply 
that the H\"older semi-norms on these spaces 
may be replaced by the standard families of the semi-norms  
\begin{equation*}
\Big\{\|\cdot\|_{C^{2s,0,s,0} _{E^i}
(\mathcal{X}_T)} \Big\}_{s\in \mathbb{Z}_+}, \quad 
\Big\{ \|\cdot\|_{C^{s,0} _{E^i}
(\mathcal{X})} \Big\}_{s\in \mathbb{Z}_+},
\end{equation*}
respectively, defining the same topologies (see, for instance, 
\cite[Chapter II, Example 1.1.4, p. 134]{Ham82}). 

\begin{corollary}
\label{c.OpenMap.psi}
The mapping 
\begin{equation}
\label{eq.heat.pseudo.nn.infty2}
   I \! + \! \varPsi_{\mu}^i \pi ^i{\mathcal N}^i :
   {C}^{\infty} _{E^i}(\mathcal{X}_T)  \cap \mathcal{S}_{(A^{i-1})^\ast}
\to  {C}^{\infty} _{E^i}(\mathcal{X}_T )\cap \mathcal{S}_{(A^{i-1})^\ast} 
\end{equation}
is injective and open. Moreover, its range is connected.
\end{corollary}

\begin{proof} 
Let $v^{(0)}  \in {C}^{\infty} _{E^i}(\mathcal{X}_T )\cap \mathcal{S}_{(A^{i-1})^\ast} $ 
be an element belonging to the range of the map \eqref{eq.heat.pseudo.nn.infty2}.

Fix $\lambda'\in (0,1)$. 
It follows from Theorem \ref{t.OpenMap.psi}  that there exists a number $\delta>0$ such 
that section 
$v^{(1)} \in {C}^{\infty} _{E^i}(\mathcal{X}_T )\cap \mathcal{S}_{(A^{i-1})^\ast}$
 satisfying 
\begin{equation} \label{eq.open.ineq}
\|v^{(1)}-v^{(0)}\|_{\mathfrak{C}^{1,\mathbf{s} (1,\lambda'/2,\lambda')}_{E^i} 
(\mathcal{X}_T)} <\delta,
\end{equation}
belongs to the range of the map \eqref{eq.heat.pseudo.nn} with $s=k=1$, i.e. 
there is a section $v \in \mathfrak{C}^{1,\mathbf{s} (1,\lambda'/2,\lambda')}
 _{E^i}(\mathcal{X}_T )$ satisfying 
\begin{equation}
\label{eq.heat.pseudo.nA}
(I  +  \varPsi_{\mu}^i\pi^i \mathcal{N}^i) v= v^{(1)}.
\end{equation}
Therefore $v \in {C}^{\infty} _{E^i}(\mathcal{X}_T)  \cap \mathcal{S}_{(A^{i-1})^\ast}$
because of the Regularity Theorem \ref{t.imp.reg}.

Finally, as the space ${C}^{\infty} _{E^i}(\mathcal{X}_T)  \cap \mathcal{S}_{(A^{i-1})^\ast}$ 
is connected. the range is connected, too,  because the map is continuous, 
which was to be proved. 
\end{proof}

\begin{corollary}  \label{c.OM.infty}
Equations \eqref{eq.NS.complex} induce a continuous, injective  mapping ${\mathfrak A}^i$
 from $ C^\infty _{E^{i}} (\mathcal{X}_T) \cap 
\mathcal{S}_{(A^{i-1})^\ast} 
\times  C^\infty _{E^{i-1}} (\mathcal{X}_T) \cap 
\mathcal{S}_{(A^{i-2})^\ast}  \ominus C^\infty([0,T],\mathcal{H}^{i-1})$ to the space 
$C^\infty _{E^{i}} (\mathcal{X}_T) 
\times C^\infty _{E^{i}} (\mathcal{X}) \cap \mathcal{S}_{(A^{i-1})^\ast}$. 
The range of this mapping is open and connected.
\end{corollary}

\begin{proof} The continuity of  map  follows from 
Lemma \ref{l.map.G}. Its injectivity follows from Uniqueness Theorem \ref{t.NS.unique}. 
 Finally, the statement on the open and connected range follows from 
Lemma \ref{l.reduce.psi.n} and Corollary \ref{c.OpenMap.psi}.
\end{proof}

According to Lemma \ref{l.reduce.psi.n}, an Existence Theorem for mapping 
\eqref{eq.heat.pseudo.nn} on the scale $\mathfrak{C}^{k,\mathbf{s} (s,\lambda,\lambda')} 
_{E^i}(\mathcal{X}_T) \cap \mathcal{D}_{A^{i}}$ immediately results on an Existence 
Theorem for the Navier-Stokes type Equations \eqref{eq.NS.complex}. Similar conclusion 
can be made on mapping \eqref{eq.heat.pseudo.nn.infty2} and 
the mappings from Corrolaries \ref{c.open.NS.short}, \ref{c.OM.infty}. 
%\eqref{eq.A.infty2}, respectively. 
Thus, Theorem \ref{t.OpenMap.psi} and Corollary 
\ref{c.OpenMap.psi} suggest a clear direction for the development of the topic:
 one can take into account the following property of the so-called clopen set.

\begin{corollary}
\label{c.clopen}
Let $s,k \in \mathbb N$  and $0 < \lambda < \lambda' < 1$. If the range of the mapping 
\eqref{eq.heat.pseudo.nn} (or \eqref{eq.heat.pseudo.nn.infty2} or 
the mappings from Corrolaries \ref{c.open.NS.short}, \ref{c.OM.infty}) %\eqref{eq.A.infty2}) 
is closed then it coincides with the whole destination space.
\end{corollary}

\begin{proof}
Since the destination space is convex, it is connected.
As is known, the only clopen (closed and open) sets in a connected topological vector space
are the empty set and the space itself. Hence, the range of the mapping $I + 
\varPsi_\mu^i \pi^i{\mathcal N}^i$ is closed if and only if it
coincides with the whole destination space. For other mappings the proof 
is similar. 
\end{proof}

As we noted in the introduction, the main motivative example of this paper is concerned with 
the Navier-Stokes type Equations associated with the de Rham complex $\{d^i, \Lambda^i \} $
over the  manifold ${\mathcal X}$, see, for instance, \cite[\S 1.2.6]{Tark95a}, where 
\begin{equation} \label{eq.M.deRham}
{\mathcal M}_{1,1} (w,u)= \star (\star w \wedge u),\,\, {\mathcal M}_{1,2} (v,u)
= \star (u\wedge \star v)/2,
\end{equation}
with $\star$ being the $\star$-Hodge operator acting on exterior differential forms. 
In this way,  taking the $3$-dimensional torus ${\mathbb T}^3$ as 
$\mathcal X$ and identifying $1$-differential forms on ${\mathbb T}^3$ with 
periodic vector fields over ${\mathbb R}^3$,  Corollaries \ref{c.open.NS.short} and 
\ref{c.OM.infty} are Open Mapping Theorems for the Navier-Stokes Equations 
in the so-called periodic setting, see \cite[Ch. 1, \S 2, \S 3]{Tema95} 
because in this situation we have 
\begin{equation*}
A^0=d^0=\nabla, \,\, A^1 = d^1=\mbox{rot}, \,\, (A^0)^* = (d^0)^* = -\mbox{div},\,\, 
(A^{-1})^*=(d^{-1})^*=0 
\end{equation*}
and  ${\mathcal N}^1 (v)$, defined by \eqref{eq.nonlinear} and \eqref{eq.M.deRham}, 
gives precisely \eqref{eq.Lamb}, see \cite{ShlTa18}. 

It is worth to note that no apriori estimates for solutions to the non-linear Navier-Stokes 
type equations were used to achieve the stability property -- it is sufficient to use the 
standard estimates for solutions to linear elliptic and parabolic equations in a uniqueness 
class for the original non-linear problem that is fit for any dimension $n\geq 2$. However, 
for the existence of  even  weak solutions to \eqref{eq.NS.complex} one should 
assume that the bilinear forms ${\mathcal M}_{i,1}$ have additional properties, see 
\cite{Lady70}, \cite{Tema79}, \cite{MTaSh19} for the properties of the so-called 
trilinear form. This means that the stability property is only a first step toward an 
Existence Theorem for regular solutions to \eqref{eq.NS.complex}.

\smallskip

\textit{Acknowledgments\,} 
The authors were supported by the grant of the Foundation for the Advancement of Theoretical 
Physics and Mathematics "BASIS"{}. We thank also Prof. 
N. Tarkhanov for an essential help in writing  section \S \ref{s.thoitwHs}.


\begin{thebibliography}{XXXXXX}

%\bibitem{Agra90} 
%{Agranovich, M.~S.},
%  \textit{Elliptic operators on closed manifold},
%  In: Current Problems of Mathematics, Fundamental Directions, Vol. 63,
%      VINITI, 1990, 5--129.

%\bibitem{Be11}
%Behrndt, T., \textit{On the Cauchy problem for the heat equation on Riemannian manifolds with 
%conical singularities}, The Quarterly Journal of Math. \textbf{64} (2011), no.~4, 981--1007.

%\bibitem{BerMaj02}
%Bertozzi, A., Majda, A., 
%\textit{Vorticity and Incompressible Flows}, Cambridge University Press, Cambridge, 2002.

%\bibitem{Bel79}
%Belonosov, V. S., \textit{Estimates of solutions of parabolic systems in weighted H\"{o}lder 
%classes and some of their applications}, Mat. Sb. \textbf{110} (1979), no.~2, 163--188.

%\bibitem{CaKoNi82}
%Caffarelli, L., Kohn, R.,  Nirenberg,  L., 
%\textit{Partial regularity of suitable weak solutions of the Navier-Stokes equations}, 
%Comm. Pure and Appl. Math. \textbf{35} (1982), 771--831.

\bibitem{ChanCzub13} %{CC13}
Chan, C.H. and Czubak, M.,
\textit{Non-Uniqueness of the Leray-Hopf solutions in hyperbolic setting}. 
Dyn. Part. Diff. Eq., 10:1 (2013), 43-77.

%\bibitem{DD12} Demengel, F., Demengel, G., \textit{Functional Spaces for the Theory
%of Elliptic Partial Differential Equations}, Universitext,
%%DOI 10.1007/978-1-4471-2807-6 2,
%Springer-Verlag London Limited 2012.

\bibitem{EbinMars70}
Ebin D. G., and Marsden, J.,
\textit{Groups of Diffeormophisms and the motion of an incompressible fluid}, 
Annals of Math. \textbf{92} (1970), 102--163.

\bibitem{Eide56}
Eidelman, S. D.,
 \textit{On fundamental solutions of parabolic equations},
 Mat. Sb. \textbf{38 (80)} (1956), no.~1, 51--92.

\bibitem{Eide90}
Eidelman, S. D.,
\textit{Parabolic equations}, Partial differential equations – 6, Itogi Nauki i Tekhniki. Ser. Sovrem. Probl. Mat. Fund. Napr., 63, VINITI, Moscow, 1990, 201--313.

\bibitem{Fri64}
Friedman, A., \textit{Partial Differential Equations of Parabolic Type},
 Prentice-Hall, Inc., Englewood Cliffs, NJ, 1964.

\bibitem{FursVish80}
Fursikov, A.~V., Vishik, M.~I.,
\textit{Mathematical Problems of Statistical Hydrodynamics}, Nauka, Moscow, 1980, 440~pp.

\bibitem{Gal13}
Gallagher, I., \textit{Remarks on the global regularity for solutions to the incompressible
Navier-Stokes equations}, European Congress of Mathematics, 331--345, Eur. Math. Soc.,
Z\"urich, 2013.

\bibitem{GiTru83} 
Gilbarg, D., Trudinger, N., \textit{Elliptic Partial Differential 
Equations of second order}, Berlin, Springer-Verlag, 1983.

%\bibitem{Gun34}
%Gunther, N. M., \textit{La th\'eorie du potentiel et ses applications aux probl\'emes 
%fondamentaux de la physique  math\'ematique}, Gauthier-Villars, Paris, 1934.

\bibitem{Ham82}
Hamilton, R. S., \textit{The inverse function theorem of Nash and Moser},
 Bull. of the AMS \textbf{7} (1982), no.~1, 65--222.

\bibitem{Hopf51}
Hopf, E.,  \textit{\"{U}ber die Anfangswertaufgabe f\"{u}r die hydrodynamischen 
Grundgleichungen}, Math. Nachr. \textbf{4} (1951), 213--231.

\bibitem{Kolm42}
Kolmogorov, A.~N., \textit{Equations of turbulent mouvement of incompressible fluid}, 
Izv. AN SSSR, Physics Series \textbf{6} (1942), no.~1, 56--58.

\bibitem{Lady70}
Ladyzhenskaya, O.~A., \textit{Mathematical Problems of Incompressible Viscous Fluid}, 
Nauka, Moscow, 1970, 288~pp.

%\bibitem{Lady03}
%Ladyzhenskaya, O.~A., \textit{The sixth prize millenium problem: Navier-Stokes equations, 
%existence and smoothness}, Russian Math. Surveys \textbf{58} (2003), no.~2, 45--78.

\bibitem{LadSoUr67}
Ladyzhenskaya, O. A., Solonnikov, V. A., and Ural'tseva, N. N., \textit{Linear and 
Quasilinear Equations of Parabolic Type}, Nauka, Moscow, 1967.

\bibitem{LaLi}
Landau, L.~D.,  and Lifshitz, E.~M.,
 \textit{Fluid Mechanics} (V.~6 of \textit{A Course of Theoretical Physics}),
 Pergamon Press, 1959.

\bibitem{Lera34a}
Leray, J.,
 \textit{Essai sur les mouvements plans d'un liquid visqueux que limitend des parois},
 J. Math. Pures Appl. \textbf{9} (1934), 331--418.

\bibitem{Lera34b}
Leray, J., \textit{Sur le mouvement plans d'un liquid visqueux emplissant l'espace},
 Acta Math. \textbf{63} (1934), 193--248.

\bibitem{Lich16}
Lichtenfelz, L., \textit{Nonuniqueness of solutions of the Navier-Stokes 
equations on Riemannian manifolds}.  
Ann. Global. Anal. Geom. 50:3 (2016), 237--248.

\bibitem{Lion61}
Lions, J.-L., \textit{\'{E}quations diff\'{e}rentielles op\'{e}rationelles et 
probl\`{e}mes aux limites}, Springer-Verlag, Berlin, 1961.

\bibitem{MTaSh19}
Mera A., Tarkhanov N., Shlapunov A.A., 
{\it Navier-Stokes Equations for Elliptic Complexes}, Journal of Siberian Federal 
University. Math. and Phys., 12:1 (2019),  3--27. 


%\bibitem{McOw79}
%Mc~Owen, R., \textit{Behavior of the Laplacian on weighted Sobolev spaces},
% Comm. Pure Appl. Math. {\bf 32} (1979), 783--795.

\bibitem{Nic07}
Nicolaescu, L.I., \textit{Lectures on the Geometry of Manifolds}, World Scientific, London, 
2007.

%\bibitem{Serr59} 
%Serrin J., 
%\textit{Mathematical Principles
%of Classical Fluid Mechanics (Encyclpedia of Physics)}. Springer-Verlag, 1959.

%\bibitem{SidShl19} 
%{Sidorova (Gagelgans), K.V., Shlapunov, A.}, \textit{On the Closure of Smooth Compactly 
%Supported Functions in Weighted Ho\"older Spaces}, Math. Notes, V. 105:4 (2019), 604--617.  

%\bibitem{SidShl20} {Gagelgans, K.V., Shlapunov, A.}, 
%\textit{On the de Rham complex on a scale of anisotropic 
%Weighted H\"older Spaces}, ???, V. ???:? (2020), ???--???.  

\bibitem{ShlTa18} 
{Shlapunov, A., Tarkhanov, N.}, \textit{An Open Mapping Theorem for the Navier-Stokes 
Equations}, Advances and Applications in Fluid Mechanics, 21:2 (2018), 127--246.

\bibitem{Sm65}
Smale, S., \textit{An infnite dimensional version of Sard's theorem},
 Amer. J. Math. \textbf{87} (1965), no.~4, 861--866.

%\bibitem{Sol64}
%Solonnikov, V. A., \textit{Estimates for solutions of a non-stationary linearized system of 
%Navier-Stokes equations}, Trudy Mat. Inst. Steklov \textbf{70} (1964), 213--317.
%
%\bibitem{Sol65}
%Solonnikov, V. A., \textit{On boundary value problems for linear parabolic systems of 
%differential equations of general form}, Trudy Mat. Inst. Steklov \textbf{83} (1965), 3--163.
%
%\bibitem{Sol06}
%Solonnikov, V. A., \textit{Estimates of the solution of model evolution generalized Stokes 
%problem in weighted H\"{o}lder spaces}, Zap. Nauchn. Sem. POMI \textbf{336} (2006), 211–-238.

\bibitem{Tark95a}
Tarkhanov, N.,  \textit{Complexes of Differential Operators}, 
Kluwer Academic Publishers, Dordrecht, NL, 1995.

\bibitem{Tayl10} 
Taylor, M., 
\textit{Partial Differential Equations III: non-linear 
equations}. Springer-Verlag, 2010.


\bibitem{Tema79}
Temam, R., \textit{Navier-Stokes Equations. Theory and Numerical Analysis},
 North Holland Publ. Comp., Amsterdam, 1979.

\bibitem{Tema95}
Temam, R., \textit{Navier-Stokes Equations and Nonlinear Functional Analysis},
2nd ed.,  SIAM, Philadelphia, 1995.

\end{thebibliography}
\end{document}